\documentclass[12pt]{amsart}
\usepackage[margin=0.5in]{geometry}
\usepackage{graphicx,epstopdf,color}
\usepackage{amssymb,amscd,amsthm,amsxtra}

\usepackage[font=footnotesize]{caption}

\usepackage{mathrsfs}
\renewcommand{\mathcal}{\mathscr}

\numberwithin{equation}{section}

\newtheorem{theorem}{Theorem}[section]
\newtheorem{lemma}[theorem]{Lemma}
\newtheorem{proposition}[theorem]{Proposition}

\newtheorem{corollary}[theorem]{Corollary}

\newcommand{\R}{\mathbb R}
\newcommand{\N}{\mathbb N}

\renewcommand{\leq}{\leqslant}
\renewcommand{\le}{\leqslant}

\renewcommand{\ge}{\geqslant}
\renewcommand{\epsilon}{\varepsilon}

\newcommand{\per}{\,{\rm Per}_s}

\begin{document}

\author[S. Dipierro]{Serena Dipierro}
\address[Serena Dipierro]{School of Mathematics and Statistics,
University of Melbourne,
Richard Berry Building,
Parkville VIC 3010,
Australia;
School of Mathematics and Statistics,
University of Western Australia,
35 Stirling Highway,
Crawley, Perth
WA 6009, Australia;
Weierstra{\ss} Institut f\"ur Angewandte
Analysis und Stochastik, Hausvogteiplatz 5/7, 10117 Berlin, Germany}
\email{serena.dipierro@ed.ac.uk}

\author[O. Savin]{Ovidiu Savin}
\address[Ovidiu Savin]{Department of Mathematics, Columbia University,
2990 Broadway,
New York NY 10027, USA.}
\email{savin@math.columbia.edu}

\author[E. Valdinoci]{Enrico Valdinoci}
\address[Enrico Valdinoci]{School of Mathematics and Statistics,
University of Melbourne,
Richard Berry Building,
Parkville VIC 3010,
Australia;
School of Mathematics and Statistics,
University of Western Australia,
35 Stirling Highway,
Crawley, Perth
WA 6009, Australia;
Weierstra{\ss} Institut f\"ur Angewandte
Analysis und Stochastik, Hausvogteiplatz 5/7, 10117 Berlin, Germany;
Dipartimento di Matematica, Universit\`a degli studi di Milano,
Via Saldini 50, 20133 Milan, Italy.}
\email{enrico@math.utexas.edu}

\title{Graph properties for nonlocal minimal surfaces}

\begin{abstract}
In this paper we show that a nonlocal minimal surface which is
a graph outside a cylinder is in fact a graph in the whole of the space.

As a consequence, in dimension~$3$, we show that the graph
is smooth.

The proofs rely on convolution techniques
and appropriate integral estimates which show the pointwise
validity of an Euler-Lagrange equation related to the nonlocal mean curvature.
\end{abstract}

\subjclass[2010]{49Q05, 35R11, 53A10}
\keywords{Nonlocal minimal surfaces, graph properties, regularity theory.}

\maketitle

\section{Introduction} 

This paper deals with the geometric properties
of the minimizers of a nonlocal perimeter functional.

More precisely, given~$s\in(0,1/2)$,
and an open set~$\Omega\subseteq\R^n$,
the $s$-perimeter of a set~$E\subseteq\R^n$ in~$\Omega$ 
was defined in~\cite{CRS}
as
$$ \per(E,\Omega) := L(E\cap\Omega, E^c)+L(\Omega\setminus E,
E\setminus\Omega),$$
where~$E^c:=\R^n\setminus E$ and, for any
disjoint sets~$F$ and~$G$,
$$ L(F,G):=\iint_{F\times G} \frac{dx\,dy}{|x-y|^{n+2s}}.$$
This nonlocal perimeter captures the global contributions
between the set~$E$ and its complement
and it is related to some models in geometry and physics,
such as the motion by nonlocal mean curvature (see~\cite{SOU})
and the phase transitions in presence of long-range interactions
(see~\cite{GAMMA}).\medskip

As customary in the calculus of
variation literature,
one says that~$E$ is $s$-minimal in a bounded open set~$\Omega$ if~$\per(E,\Omega)<+\infty$
and~$\per(E,\Omega)\le \per(F,\Omega)$ among all the sets~$F$
which coincide with~$E$ outside~$\Omega$
(with a slight abuse of language, when~$\Omega$
is unbounded, we say that~$E$ is $s$-minimal
in~$\Omega$ if it is $s$-minimal
in any bounded open subsets of~$\Omega$, see also~\cite{LOM}
for further details).\medskip

Clearly, the above definition does not distinguish, in principle,
between sets whose symmetric difference has vanishing Lebesgue measure.
As in the case of minimal surfaces, we will implicitly
identify sets up to sets of zero measure.
More precisely, whenever necessary, we will
take a representation of the set
which coincides with
the ``measure theoretic
interior''
of the original set, i.e.
those points for which there exists a ball of positive radius around them
which is included in the set, up to sets of measure zero
(we refer to Appendix~\ref{9udfdtfd6yh665544}
at the end of this paper for a detailed discussion).\medskip

Several analytic and geometric properties of $s$-minimal sets
have been recently investigated, in terms,
for instance, of asymptotics~\cite{MAZYA, BREZIS, CV, MARTIN, DFPV},
regularity~\cite{CV-REG, SV-REG, figalli}
and classification~\cite{FALL-CABRE, CIRAOLO}. Some examples of
$s$-minimal sets (or, more generally, of sets which
possess vanishing nonlocal mean curvatures) have been given
in~\cite{LAWSON, NOSTRO}.
\medskip

The main result of this paper establishes that an $s$-minimal set is
a subgraph, if so are its exterior data:

\begin{theorem}\label{GRAPH}
Let~$\Omega_o$ be an open and bounded subset
of~$\R^{n-1}$ with boundary of class~$C^{1,1}$,
and let~$\Omega:=\Omega_o\times\R$.
Let~$E$ be an $s$-minimal set in~$\Omega$.
Assume that
\begin{equation}\label{HJ:GRAPH:2}
E\setminus\Omega = \{ x_n < u(x'),\; x'\in\R^{n-1}\setminus\Omega_o \},
\end{equation}
for some continuous function~$u:\R^{n-1}\to\R$.
Then
\begin{equation}\label{p98T78-00gra}
E\cap \Omega = \{ x_n < v(x'),\; x'\in\Omega_o \},
\end{equation}
for some~$v:\R^{n-1}\to\R$, where~$v$ is a uniformly
continuous 
function
in~${\Omega_o}$, with~$v=u$ outside~$\Omega_o$.
\end{theorem}

We mention that, in general,
$s$-minimal surfaces are not continuous up to the boundary
of the domain (even if the datum outside is smooth),
and indeed boundary stickiness phenomena occur (see~\cite{NOSTRO}
for concrete examples). The possible discontinuity at the boundary
makes the proof of Theorem~\ref{GRAPH}
quite delicate, since the graph property ``almost fails''
in a cylinder (see Theorem~1.2 in~\cite{NOSTRO}),
and, in general, the graph property
cannot be deduced only from the outside data
but it may also depend on the regularity of the domain.

More precisely, we stress that, given the examples
in \cite{NOSTRO},
even in simple cases, boundary
stickiness of nonlocal minimal surfaces
has to be expected,
therefore the 
Dirichlet problem for nonlocal minimal surfaces
cannot be solved in the class of functions that are
continuous up to the boundary (even for smooth
Dirichlet data).
Nevertheless, Theorem~\ref{GRAPH} says that one can solve
the Dirichlet problem in the class of (not necessarily continuous)
graphs.

Due to the boundary stickiness (i.e., due to
the possibly discontinuous behavior of the solution
of the Dirichlet problem),
the proof of Theorem~\ref{GRAPH} cannot be
just a simple generalization of the proof of similar results
for classical minimal surfaces.
As a matter of fact, 
in principle, once the surface reaches a vertical
tangent, it might be possible that it bends in the ``wrong'' direction,
which would
make the graph property false (our result in
Theorem~\ref{GRAPH} excludes exactly this
possibility).

The key point for this
is that, due to the study of the obstacle problem for nonlocal minimal
surfaces treated
in \cite{PRO}, we more or less understand how the minimal surfaces
separate from
the boundary of the cylinder if a vertical piece is present.

In addition, since we are not assuming any smoothness of the surface
to start with, some care is needed to compute the fractional
mean curvature in a pointwise sense.
\medskip

The proof of Theorem~\ref{GRAPH}
is based on a sliding method, but some (both technical
and conceptual)
modifications are needed to make the classical argument work,
due to the contributions ``coming from far''.
We stress that, since the $s$-minimal set is not assumed to be
smooth, some supconvolutions techniques are needed to
take care of interior contact points.
Moreover, a fine analysis of the possible contact points which lie
on the boundary (and at infinity)
is needed to complete the arguments.
\medskip

As a matter of fact, we think that it is an interesting open
problem to determine whether or not Theorem~\ref{GRAPH}
holds true without the assumption
that~$\partial\Omega_o$ is of class~$C^{1,1}$
(for instance, whether or not a similar statement holds
by assuming only that~$\partial\Omega_o$ is
Lipschitz). We remark that, under weak regularity assumptions,
one cannot make use of the results in \cite{PRO}.
\medskip

The results in Theorem~\ref{GRAPH}
may be strengthen in the case of dimension~$3$, by proving
that two-dimensional minimal graphs
are smooth. Indeed, we have:

\begin{theorem}\label{coro:reg}
Let~$\Omega_o$ be an open and bounded subset
of~$\R^{2}$ with boundary of class~$C^{1,1}$,
and let~$\Omega:=\Omega_o\times\R$.
Let~$E$ be an $s$-minimal set in~$\Omega$.
Assume that
\begin{equation*}
E\setminus\Omega = \{ x_n < u(x'),\; x'\in\R^{n-1}\setminus\Omega_o \},
\end{equation*}
for some continuous function~$u:\R^{n-1}\to\R$.
Then
\begin{equation}\label{aJk:hg1}
E\cap \Omega = \{ x_3 < v(x'),\; x'\in\Omega_o \},
\end{equation}
for some~$v\in C^\infty(\Omega_o)$.
\end{theorem}

The proof of Theorem~\ref{coro:reg}
relies on Theorem~\ref{GRAPH} and on a 
Bernstein-type result of~\cite{figalli}.
\medskip

The rest of the paper is organized as follows.
In Section~\ref{SEC:SUBCO}
we discuss the 
notion of supconvolutions and subconvolutions
for a nonlocal minimal surface, presenting the geometric
and analytic properties that we need for the proof of
Theorem~\ref{GRAPH}.

In Section \ref{KJ56DFfSE}
we collect a series of
auxiliary results needed to compute suitable integral
contributions and obtain an appropriate fractional mean
curvature equation in a pointwise sense
(i.e., not only in the sense of viscosity, as done
in the previous literature).

The proof of
Theorem~\ref{GRAPH} is given in Section~\ref{HY:08}
and the proof of Theorem \ref{coro:reg}
is given in Section~\ref{LJrtyuoo}.

Finally, in Appendix~\ref{9udfdtfd6yh665544}, we
observe that we can modify the $s$-minimal set~$E$
by a set of measure zero in order to identify~$E$
with the set of its interior points in the measure
theoretic sense.

\section{Supconvolution and subconvolution of a set}\label{SEC:SUBCO}

In this section, we introduce the notion of supconvolution
and discuss its basic properties. This is the nonlocal modification
of a technique developed in~\cite{CORDOBA} for the local case.

Given~$\delta>0$, we define the supconvolution
of the set~$E\subseteq\R^n$ by
$$ E^\sharp_\delta := \bigcup_{x\in E} \overline{ B_\delta(x) }.$$

\begin{lemma}\label{ConC1} We have that
$$ E^\sharp_\delta = \bigcup_{ {v\in\R^n}\atop{|v|\le\delta} } (E+v).$$
\end{lemma}

\begin{proof} Let~$y\in \overline{B_\delta(x)}$, with~$x\in E$.
Let~$v:=y-x$. Then~$|v|\le\delta$ and~$y=x+v\in E+v$,
and one inclusion is proved.

Viceversa, let now~$y\in E+v$, with~$|v|\le\delta$.
We set in this case~$x:=y-v$. Hence~$|y-x|=|v|\le\delta$,
thus~$y\in \overline{B_\delta(x)}$. In addition,~$x\in (E+v)-v=E$,
so the other inclusion is proved.
\end{proof}

\begin{corollary}\label{Cor-CONV}
If~$p\in\partial E^\sharp_\delta$, then there exist~$v\in\R^n$,
with~$|v|=\delta$, and~$x_o\in \partial E$ such that~$p=x_o+v$
and~$B_\delta(x_o)\subseteq E^\sharp_\delta$.

Also, if~$E^\sharp_\delta$ is touched from the outside at~$p$
by a ball~$B$, then~$E$ is touched from the outside at~$x_o$
by~$B-v$.
\end{corollary}

\begin{proof} Since~$p\in\overline{E^\sharp_\delta}$,
we have that there exists a sequence~$p_j\in E^\sharp_\delta$
such that~$p_j\to p$ as~$j\to+\infty$.
By Lemma~\ref{ConC1}, we have that~$p_j\in E+v_j$, for some~$v_j\in\R^n$
with~$|v_j|\le\delta$. That is, there exists~$x_j\in E$ such that~$p_j=
x_j+v_j$. By compactness, up to a subsequence we may assume that~$v_j\to v$
as~$j\to+\infty$, for some~$v\in\R^n$
with
\begin{equation}\label{Ofg6799}
|v|\le\delta.\end{equation} Therefore
\begin{equation}\label{Ofg6799-BIS}
x_j= p_j-v_j\to p-v =: x_o\end{equation}
as $j\to+\infty$.
By construction, 
\begin{equation}
\label{PDSA1} x_o\in\overline{E}\end{equation}
and
\begin{equation}
\label{PDSA2} p=x_o+v.\end{equation}
Now we show that
\begin{equation}\label{SA45fhhH}
x_o \in \overline{E^c}.\end{equation}
For this, since~$p\in\overline{\R^n\setminus E^\sharp_\delta}$,
we have that there exists a sequence~$q_j\in \R^n\setminus E^\sharp_\delta$
such that~$q_j\to p$ as~$j\to+\infty$.

Notice that
\begin{equation}\label{KJgh67pp99}
\overline{B_\delta(q_j)}\cap E=\varnothing.\end{equation}
Indeed, if not, we would have that there exists~$z_j\in\overline{ B_\delta(q_j)}\cap E$.
So we can define~$w_j:=q_j-z_j$. We see that~$|w_j|\le\delta$
and therefore~$q_j=z_j+w_j\in E+w_j\subseteq E^\sharp_\delta$,
which is a contradiction.

Having established~\eqref{KJgh67pp99}, we use it to deduce that~$q_j-v_j
\in E^c$. Thus passing to the limit
$$ x_o=p-v=\lim_{j\to+\infty} q_j-v_j\in
\overline{E^c}.$$
This proves~\eqref{SA45fhhH}.

{F}rom \eqref{PDSA1} and~\eqref{SA45fhhH},
we conclude that
\begin{equation}\label{BGf-1k}
x_o\in\partial E.\end{equation}
Now we show that
\begin{equation}\label{k7800a}
|v|=\delta.
\end{equation}
To prove it, suppose not. Then, by~\eqref{Ofg6799},
we have that~$|v|<\delta$. That is, there exists~$a\in(0,\delta)$
such that~$|v|<\delta-a$. Then, by~\eqref{Ofg6799-BIS},
$$ |x_j-p|\le |x_j-x_o|+|x_o-p|=
|x_j-x_o| +|v|<\delta-\frac{a}{2},$$
if~$j$ is large enough. Hence~$B_{a/2}(p)\subseteq B_\delta(x_j)
\subseteq E^\sharp_\delta$, that says that~$p$ lies
in the interior of~$E^\sharp_\delta$. This is in contradiction
with the assumptions of
Corollary~\ref{Cor-CONV}, and so~\eqref{k7800a} is proved.

Now we claim that
\begin{equation}\label{Ju753110999}
B_\delta(x_o)\subseteq E^\sharp_\delta.
\end{equation}
To prove this, let~$z\in B_\delta(x_o)$.
Then, $|z-x_o|\le \delta-b$, for some~$b\in(0,\delta)$.
Accordingly, by~\eqref{Ofg6799-BIS}, we have that~$|z-x_j|\le \delta-\frac{b}{2}$
if~$j$ is large enough. Hence~$z\in B_\delta(x_j)\subseteq E^\sharp_\delta$.
This proves~\eqref{Ju753110999}.

Thanks to~\eqref{PDSA2},
\eqref{BGf-1k}, \eqref{k7800a} and~\eqref{Ju753110999},
we have completed the proof of the
first claim in the statement of Corollary~\ref{Cor-CONV}.

Now, to prove the
second claim in the statement of Corollary~\ref{Cor-CONV},
let us consider a ball~$B$ such that~$B\subseteq
\R^n\setminus E^\sharp_\delta$ and~$p\in\partial B$.
Then $x_o=p-v\in(\partial B)-v=\partial (B-v)$.
Moreover,
$$ B-v \subseteq
(\R^n\setminus E^\sharp_\delta)-v =
\R^n\setminus (E^\sharp_\delta -v).$$
Since~$E\subseteq E^\sharp_\delta$, we have that
$$ \R^n\setminus (E^\sharp_\delta -v) \subseteq \R^n\setminus (E-v).$$
Consequently, we obtain that~$B-v \subseteq \R^n\setminus (E-v)$,
which completes the proof
of the second claim of Corollary~\ref{Cor-CONV}.\end{proof}

The supconvolution has an important property
with respect to the fractional mean curvature,
as stated in the next result:

\begin{lemma}\label{AXr5yy}
Let~$p\in\partial E^\sharp_\delta$, $v\in\R^n$
with~$|v|\le\delta$ and~$x_o\in \partial E$ such that~$p=x_o+v$. Then
$$ \int_{\R^n} \frac{\chi_{E^\sharp_\delta}(y)-
\chi_{\R^n\setminus E^\sharp_\delta}(y)}{|p-y|^{n+2s}}\,dy
\ge 
\int_{\R^n} \frac{\chi_{E}(y)- \chi_{\R^n\setminus E}(y)}{|x_o-y|^{n+2s}}\,dy
.$$\end{lemma}

\begin{proof} The claim follows simply
by the fact that~$E^\sharp_\delta\supseteq E+v$
and the translation invariance of the fractional mean curvature.
\end{proof}

\begin{corollary}\label{SufT67HHJ7}
Let~$E$ be an $s$-minimal set in~$\Omega$.
Let~$p\in\partial E^\sharp_\delta$. Assume that~$\overline{B_\delta(p)}
\subseteq\Omega$ and that~$E^\sharp_\delta$ is touched from the outside at~$p$
by a ball. Then
$$ \int_{\R^n} \frac{\chi_{E^\sharp_\delta}(y)-
\chi_{\R^n\setminus E^\sharp_\delta}(y)}{|p-y|^{n+2s}}\,dy
\ge0.$$
\end{corollary}

\begin{proof} By Corollary~\ref{Cor-CONV}, we know that
there exist~$v\in\R^n$
with~$|v|\le\delta$ and~$x_o\in \partial E$ such that~$p=x_o+v$,
and that~$E$ is touched by a ball from the outside at~$x_o$.

We remark that~$x_o\in \overline{B_\delta(p)}
\subseteq\Omega$. So, we can use
the Euler-Lagrange equation in the viscosity sense
(see Theorem~5.1 in~\cite{CRS}) and obtain that
$$ \int_{\R^n} \frac{\chi_{E}(y)-
\chi_{\R^n\setminus E}(y)}{|x_o-y|^{n+2s}}\,dy
\ge0.$$
This and Lemma~\ref{AXr5yy} give the desired result.
\end{proof}

The counterpart of the notion of
supconvolution is given by the notion of subconvolution.
That is, we define
$$ E^\flat_\delta := \R^n\setminus 
\big( (\R^n\setminus E)^\sharp_\delta\big).$$
In this setting, we have:

\begin{proposition}\label{TOUCH:omega}
Let~$E$ be an $s$-minimal set in~$\Omega$.
Let~$p\in\partial E^\sharp_\delta$. Assume that~$\overline{B_\delta(p)}
\subseteq\Omega$.

Assume also that $E^\sharp_\delta$ is touched from above
at~$p$ by a translation of~$E^\flat_\delta$, i.e. there exists~$\omega
\in\R^n$ such that~$E^\sharp_\delta\subseteq E^\flat_\delta+\omega$
and~$p\in(\partial E^\sharp_\delta)\cap
\big(\partial(E^\flat_\delta+\omega)\big)$.

Then~$E^\sharp_\delta=E^\flat_\delta+\omega$.
\end{proposition}

\begin{proof} Notice that
$$ p\in\partial
(E^\flat_\delta+\omega)
=\partial E^\flat_\delta +\omega =
\partial \big( (\R^n\setminus E)^\sharp_\delta\big) +\omega.$$
Accordingly, by the first claim in
Corollary \ref{Cor-CONV} (applied to the set~$\R^n\setminus E$
and to the point~$p-\omega$), we see that
there exist~$\tilde v\in\R^n$,
with~$|\tilde v|=\delta$, and~$\tilde x_o\in \partial (\R^n\setminus E)=\partial E$
such that~$p-\omega =\tilde x_o+\tilde v$
and~$B_\delta(\tilde x_o)\subseteq (\R^n\setminus E)^\sharp_\delta$.
That is, the set~$(\R^n\setminus E)^\sharp_\delta$
is touched from the inside at~$p-\omega$ by a ball of radius~$\delta$.
Taking the complementary set and translating by~$\omega$,
we obtain that~$E^\flat_\delta+\omega$ is touched
from the outside at~$p$ by a ball of radius~$\delta$.

Then, since~$E^\flat_\delta+\omega\supseteq E^\sharp_\delta$,
we obtain that also~$E^\sharp_\delta$
is touched
from the outside at~$p$ by a ball of radius~$\delta$.
Thus, making use of Corollary~\ref{SufT67HHJ7}, we deduce that
\begin{equation} \label{78hFFGK-094124}
\int_{\R^n} \frac{\chi_{E^\sharp_\delta}(y)-
\chi_{\R^n\setminus E^\sharp_\delta}(y)}{|p-y|^{n+2s}}\,dy
\ge0.\end{equation}
Moreover, by Corollary~\ref{Cor-CONV}, we know that~$
E^\sharp_\delta$ is touched from the inside at~$p$ by a ball of
radius~$\delta$.
By inclusion of sets, this gives that~$E^\flat_\delta+\omega$
is touched from the inside at~$p$ by a ball of
radius~$\delta$. Taking complementary sets,
we obtain that~$(\R^n\setminus E)^\sharp_\delta$
is touched from the outside at~$p-\omega$ by a ball of
radius~$\delta$.
Therefore, we can use
Corollary~\ref{SufT67HHJ7} (applied here to the set~$(\R^n\setminus E)^\sharp_\delta$),
and get that
\begin{eqnarray*}
&& 0 \le \int_{\R^n} \frac{\chi_{(\R^n\setminus E)^\sharp_\delta}(y)-
\chi_{\R^n\setminus \big( (\R^n\setminus E)^\sharp_\delta\big)}(y)}{
|p-\omega-y|^{n+2s}}\,dy\\
&&\qquad=
\int_{\R^n} \frac{\chi_{\R^n\setminus E^\flat_\delta} (y)-
\chi_{E^\flat_\delta}(y)}{
|p-\omega-y|^{n+2s}}\,dy
= -
\int_{\R^n} \frac{\chi_{E^\flat_\delta+\omega} (y)-
\chi_{\R^n\setminus (E^\flat_\delta+\omega)}(y)}{
|p-y|^{n+2s}}\,dy.\end{eqnarray*}
By comparing this estimate with the one in~\eqref{78hFFGK-094124},
we obtain that
$$ \int_{\R^n} \frac{\chi_{E^\sharp_\delta}(y)-
\chi_{\R^n\setminus E^\sharp_\delta}(y)}{|p-y|^{n+2s}}\,dy
\ge0\ge
\int_{\R^n} \frac{\chi_{E^\flat_\delta+\omega}(y)-
\chi_{\R^n\setminus (E^\flat_\delta+\omega)}(y)}{|p-y|^{n+2s}}\,dy
.$$
Since~$E^\sharp_\delta$ lies in~$E^\flat_\delta+\omega$,
the inequality above implies that the two sets must coincide.
\end{proof}

A useful variation of Proposition~\ref{TOUCH:omega}
consists in taking into account the possibility that the inclusion of the sets
only occurs inside a suitable domain.
For this, we define the cylinder
\begin{equation}\label{CYLINDER:DEF}
{\mathcal{C}}_R:= \{ x=(x',x_n)\in\R^n {\mbox{ s.t. }} |x'|<R\}.\end{equation}
Also, given $\eta>0$, we 
consider the set $\Omega_\eta$ of points which lie inside the domain
at distance greater than $\eta$ from the boundary, namely we
set
\begin{equation}\label{Oema}
\Omega_\eta := \{x\in\Omega {\mbox{ s.t. }} B_\eta(x)\subseteq\Omega\}.\end{equation}
Since $\Omega_o$ is open and bounded, 
we may suppose that $\Omega_o \subset B_{R_o}$, for some $R_o>0$,
hence $\Omega\subset{\mathcal{C}}_{R_o}$. So, for any $R>R_o$, we define
\begin{equation}\label{Oema2}
{\mathcal{D}}_{R,\eta}:=
{\mathcal{C}}_R \setminus (\Omega\setminus \Omega_{2\eta}) = \Omega_{2\eta}\cup
({\mathcal{C}}_R \setminus \Omega).
\end{equation}
With this notation, we have:

\begin{proposition}\label{TOUCH:omega-bis}
Let~$R>4(R_o+1)$ and~$\delta$, $\eta\in(0,1)$.
Let~$E$ be an $s$-minimal set in~$\Omega$.
Let~$p\in\partial E^\sharp_\delta $. Assume that
\begin{equation}\label{6yhcGGHJKoo}
\overline{B_{4(\delta+\eta)}(p)}
\subseteq\Omega_{4\eta}.\end{equation}
Assume also that $E^\sharp_\delta$ is touched in~${\mathcal{D}}_{R,\eta}$ from above
at~$p$ by a vertical
translation of~$E^\flat_\delta$, i.e. there exists~$\omega=(\omega',0)
\in\R^n$ such that~$E^\sharp_\delta\cap {\mathcal{D}}_{R,\eta}
\subseteq (E^\flat_\delta+\omega)\cap {\mathcal{D}}_{R,\eta}$
and~$p\in(\partial E^\sharp_\delta)\cap
\big(\partial(E^\flat_\delta+\omega)\big)$.

Then, for $\eta$ sufficiently small,
$$ \int_{{\mathcal{D}}_{R,\eta}} \frac{\chi_{(E^\flat_\delta+\omega)
\setminus E^\sharp_\delta}(y)-
\chi_{E^\sharp_\delta\setminus(E^\flat_\delta+\omega)
}(y)}{|p-y|^{n+2s}}\,dy\le C\left( R^{-2s} + \frac{
\eta}{
\big( {\rm dist}\,(p,\partial\Omega)
\big)^{n+2s}}\right),$$
for some~$C>0$, independent of~$\delta$, $\eta$ and~$R$.
\end{proposition}

\begin{proof} The proof is a measure theoretic version 
of the one in Proposition~\ref{TOUCH:omega}. We give the full details
for the convenience of the reader.

Notice that
$$ p\in\partial
(E^\flat_\delta+\omega)
=\partial E^\flat_\delta +\omega =
\partial \big( (\R^n\setminus E)^\sharp_\delta\big) +\omega.$$
Accordingly, by the first claim in
Corollary \ref{Cor-CONV} (applied to the set~$\R^n\setminus E$
and to the point~$p-\omega$), we see that
there exist~$\tilde v\in\R^n$,
with~$|\tilde v|=\delta$, and~$\tilde x_o\in \partial (\R^n\setminus E)=\partial E$
such that~$p-\omega =\tilde x_o+\tilde v$
and~$B_\delta(\tilde x_o)\subseteq (\R^n\setminus E)^\sharp_\delta$.
That is, the set~$(\R^n\setminus E)^\sharp_\delta$
is touched from the inside at~$p-\omega$ by a ball of radius~$\delta$.
Taking the complementary set and translating by~$\omega$,
we obtain that~$E^\flat_\delta+\omega$ is touched
from the outside at~$p$ by a ball of radius~$\delta$.
Notice also that, in view of \eqref{6yhcGGHJKoo},
such ball lies in $\Omega_{4\eta}$, which in turn lies in ${\mathcal{D}}_{R,\eta}$.

Then, since~$(E^\flat_\delta+\omega)\cap{\mathcal{D}}_{R,\eta}
\supseteq E^\sharp_\delta\cap {\mathcal{D}}_{R,\eta}$,
we obtain that also~$E^\sharp_\delta$
is touched
from the outside at~$p$ by a ball of radius~$\delta$.
Thus, making use of Corollary~\ref{SufT67HHJ7}, we deduce that
\begin{equation} \label{78hFFGK-094124-BIS}
\int_{\R^n} \frac{\chi_{E^\sharp_\delta}(y)-
\chi_{\R^n\setminus E^\sharp_\delta}(y)}{|p-y|^{n+2s}}\,dy
\ge0.\end{equation}
Moreover, by Corollary~\ref{Cor-CONV}, we know that~$
E^\sharp_\delta$ is touched from the inside at~$p$ by a ball of
radius~$\delta$.
By inclusion of sets, this gives that~$E^\flat_\delta+\omega$
is touched from the inside at~$p$ by a ball of
radius~$\delta$. Taking complementary sets,
we obtain that~$(\R^n\setminus E)^\sharp_\delta$
is touched from the outside at~$p-\omega$ by a ball of
radius~$\delta$.
Therefore, we can use
Corollary~\ref{SufT67HHJ7} (applied here to the set~$(\R^n\setminus E)^\sharp_\delta$),
and get that
\begin{eqnarray*}
&& 0 \le \int_{\R^n} \frac{\chi_{(\R^n\setminus E)^\sharp_\delta}(y)-
\chi_{\R^n\setminus \big( (\R^n\setminus E)^\sharp_\delta\big)}(y)}{
|p-\omega-y|^{n+2s}}\,dy\\
&&\qquad=
\int_{\R^n} \frac{\chi_{\R^n\setminus E^\flat_\delta} (y)-
\chi_{E^\flat_\delta}(y)}{
|p-\omega-y|^{n+2s}}\,dy
= -
\int_{\R^n} \frac{\chi_{E^\flat_\delta+\omega} (y)-
\chi_{\R^n\setminus (E^\flat_\delta+\omega)}(y)}{
|p-y|^{n+2s}}\,dy.\end{eqnarray*}
By comparing this estimate with the one in~\eqref{78hFFGK-094124-BIS},
we obtain that
$$ \int_{\R^n} \frac{\chi_{E^\sharp_\delta}(y)-
\chi_{\R^n\setminus E^\sharp_\delta}(y)}{|p-y|^{n+2s}}\,dy
\ge0\ge
\int_{\R^n} \frac{\chi_{E^\flat_\delta+\omega}(y)-
\chi_{\R^n\setminus (E^\flat_\delta+\omega)}(y)}{|p-y|^{n+2s}}\,dy
.$$
Since~$E^\sharp_\delta\cap {\mathcal{D}}_{R,\eta}$ lies in~$(E^\flat_\delta+\omega)\cap {\mathcal{D}}_{R,\eta}$,
the inequality above implies that 
\begin{equation}\label{yu6schjh0}
\begin{split}
& \int_{{\mathcal{D}}_{R,\eta}} \frac{\chi_{(E^\flat_\delta+\omega)
\setminus E^\sharp_\delta}(y)-
\chi_{E^\sharp_\delta\setminus(E^\flat_\delta+\omega)
}(y)}{|p-y|^{n+2s}}\,dy
\le
2\int_{\R^n\setminus {\mathcal{D}}_{R,\eta}} \frac{dy}{|p-y|^{n+2s}}\\ &\qquad=
2\int_{\R^n\setminus {\mathcal{C}}_R} \frac{dy}{|p-y|^{n+2s}}
+
2\int_{\Omega\setminus \Omega_{2\eta}} \frac{dy}{|p-y|^{n+2s}}
.\end{split}\end{equation}
Notice now that,
if~$y\in\R^n\setminus {\mathcal{C}}_R$, then~$|p-y|\ge |p'-y'|\ge |y'|-|p'|\ge R-R_o \ge R/2$.
Hence changing variable~$\zeta:= p-y$, we have
\begin{equation}\label{yu6schjh1}
\int_{\R^n\setminus {\mathcal{C}}_R} \frac{dy}{|p-y|^{n+2s}}
\le \int_{\R^n\setminus B_{R/2}} \frac{d\zeta}{|\zeta|^{n+2s}}\le CR^{-2s},\end{equation}
for some $C>0$.
Moreover, using again \eqref{6yhcGGHJKoo},
we see that ${\rm dist}\,(p,\partial\Omega)\ge 4\eta$.
Hence, if $y\in \Omega\setminus\Omega_{2\eta}$, we have that
$$ |p-y|\ge {\rm dist}\,(p,\partial\Omega) -2\eta\ge \frac{
{\rm dist}\,(p,\partial\Omega)}{2}.$$
As a consequence,
\begin{equation}\label{7ytrfdgJJa}
\int_{\big(\Omega\setminus \Omega_{2\eta}\big)\cap \{ |p_n-y_n|\le1\}}
\frac{dy}{|p-y|^{n+2s}}
\le \frac{
C\,\eta}{
\big( {\rm dist}\,(p,\partial\Omega)
\big)^{n+2s}}.
\end{equation}
On the other hand,
\begin{equation*}\begin{split}&
\int_{\big(\Omega\setminus \Omega_{2\eta}\big)\cap \{ |p_n-y_n|>1\}}
\frac{dy}{|p-y|^{n+2s}}
\le 
\int_{\big(\Omega\setminus \Omega_{2\eta}\big)\cap \{ |p_n-y_n|>1\}}
\frac{dy}{|p_n-y_n|^{n+2s}} \\ &\qquad\le C\eta
\int_{\{ |p_n-y_n|>1\}}
\frac{dy_n}{|p_n-y_n|^{n+2s}} \le C\eta,\end{split}
\end{equation*}
for some $C>0$ (possibly different from step to step).
The latter estimate and \eqref{7ytrfdgJJa}
imply that
$$ \int_{\Omega\setminus \Omega_{2\eta}}
\frac{dy}{|p-y|^{n+2s}}
\le 
\frac{
C\,\eta}{
\big( {\rm dist}\,(p,\partial\Omega)
\big)^{n+2s}},$$
up to renaming $C$.
By inserting this and \eqref{yu6schjh1} into \eqref{yu6schjh0},
we obtain the desired result.
\end{proof}

\section{Auxiliary integral computations and a pointwise version
of the Euler-Lagrange equation}\label{KJ56DFfSE}

We collect here some technical results, which are used during
the proofs of the main results. First, 
we recall an explicit
estimate on the weighted measure of a set trapped between two
tangent balls.

\begin{lemma}\label{Trap}
For any~$R>0$ and~$\lambda\in(0,1]$, let
$$ P_{R,\lambda}:= \big\{x=(x',x_n)\in\R^n
{\mbox{ s.t. }} |x'|\le \lambda R {\mbox{ and }} |x_n|\le R-
\sqrt{R^2-|x'|^2} \big\}. $$
Then
$$ \int_{P_{R,\lambda}} \frac{dx}{|x|^{n+2s}}
\le \frac{CR^{-2s} \lambda^{1-2s}}{1-2s},$$
for some~$C>0$ only depending on~$n$.
\end{lemma}

\begin{proof} By scaling~$y:=x/R$, we see that
$$ \int_{P_{R,\lambda}} \frac{dx}{|x|^{n+2s}}
= R^{-2s} \int_{P_{1,\lambda}} \frac{dy}{|y|^{n+2s}},$$
so it is enough to prove the desired claim for~$R=1$.

To this goal, we observe that, if~$\rho\in [0,1]$ then
$$ 1-\sqrt{1-\rho^2}\le C\rho^2,$$
for some~$C>0$ (independent of~$n$ and~$s$).
Therefore
\begin{equation}\label{90IOPP}
\int_0^\lambda \frac{1-\sqrt{1-\rho^2}}{\rho^{2+2s}}\,d\rho
\le \frac{C \lambda^{1-2s}}{1-2s},\end{equation}
up to renaming~$C>0$.

In addition, using polar coordinates in~$\R^{n-1}$
(and possibly renaming constants which only depend on~$n$),
we have
\begin{eqnarray*}
&& \int_{P_{1,\lambda}} \frac{dx}{|x|^{n+2s}}
\le \int_{P_{1,\lambda}} \frac{dx}{|x'|^{n+2s}}
=C\int_{\{ |x'|\le\lambda\}}
\left( \int_0^{1-\sqrt{1-|x'|^2}} \frac{dx_n}{|x'|^{n+2s}}\right)\,dx'\\
&&\qquad
=C\int_{\{ |x'|\le\lambda\}}
\frac{1-\sqrt{1-|x'|^2}}{|x'|^{n+2s}} \,dx'
=C\int_0^\lambda
\frac{1-\sqrt{1-\rho^2}}{\rho^{2+2s}} \,d\rho.\end{eqnarray*}
This and~\eqref{90IOPP} yield the desired result.
\end{proof}

A variation of Lemma~\ref{Trap} deals with the
case of trapping between two hypersurfaces, as stated in the following result:

\begin{lemma}\label{TRAP-C1}
Let~$C_o>0$ and~$\alpha>2s$.
For any~$L>0$, let
$$ P_{L}:= \big\{x=(x',x_n)\in\R^n
{\mbox{ s.t. }} |x'|\le L {\mbox{ and }} |x_n|\le C_o\,|x'|^{1+\alpha}\big\}. $$
Then
$$ \int_{P_{L}} \frac{dx}{|x|^{n+2s}}
\le \frac{C\,C_o\,L^{\alpha-2s } 
}{\alpha-2s},$$
for some~$C>0$ only depending on~$n$.
\end{lemma}

\begin{proof} Using polar coordinates in~$\R^{n-1}$, we have
\begin{equation*}
\int_{P_{L}} \frac{dx}{|x|^{n+2s}}\le\int_{\{|x'|\le L\}}
\left( \int_{\{ |x_n|\le C_o\,|x'|^{1+\alpha}\}} \frac{dx_n}{|x'|^{n+2s}}
\right)\,dx'
=
\int_{\{|x'|\le L\}}
\frac{2C_o\,|x'|^{1+\alpha}}{|x'|^{n+2s}}\,dx'
= \frac{C\,C_o\,L^{\alpha-2s} }{\alpha-2s},
\end{equation*}
for some~$C>0$.
\end{proof}

Now we show that
an $s$-minimal set does not have spikes going to infinity:

\begin{lemma}\label{1212}
Let~$\Omega_o$ be an open and bounded subset
of~$\R^{n-1}$ 
and let~$\Omega:=\Omega_o\times\R$.
Let~$E$ be an $s$-minimal set in~$\Omega$.

Assume that 
\begin{equation}\label{jnon8a9} 
E \setminus\Omega\, \subseteq\, \{x_n \le v(x')\},
\end{equation} 
for some~$v:\R^{n-1}\to\R$, and that, for any~$R>0$,
$$ M_R:=\sup_{|x'|\le R} v(x') <+\infty.$$
Then
\begin{equation*}
E\cap \Omega\,
\subseteq\,\{x_n\le M\}
\end{equation*}
for some~$M\in\R$ (which may depend on $s$, $n$, $\Omega_o$
and~$v$).
\end{lemma}

\begin{proof} Assume that~$\Omega_o\subseteq \{|x'|<R_o\}$,
for some~$R_o$ and let~$R>R_o+1$, to be chosen suitably large.
We show that
\begin{equation}\label{TG:1}
E\subseteq \left\{ x_n\le 2M_{5R}+\frac{3}{2}R\right\}.\end{equation}
For any~$t\ge 2M_{5R} +2R$ we slide a ball centered at~$\{x_n=t\}$
of radius~$R/2$
``from left to right''. For this, we observe that
\begin{equation}\label{gTGhG:87}
B_{R/2} (-2R,0,\dots,0,t)\subseteq E^c.
\end{equation}
Indeed, if~$x\in B_{R/2} (-2R,0,\dots,0,t)$, then
$$ \big| |x'|-2R\big| = \big| |x'|-|(-2R,0,\dots,0)|\big|
\le \big| x'-(-2R,0,\dots,0)\big|
\le \big| x-(-2R,0,\dots,0,t)\big|\le\frac{R}2.$$
In particular, 
$$ |x'|\in (R,\, 3R).$$
In addition,
$$ x_n \ge t-\frac{R}2 \ge 2M_{5R} +2R -\frac{R}2>2M_{5R} \ge v(x').$$
These considerations and~\eqref{jnon8a9} imply that~$x\in E^c$,
thus establishing~\eqref{gTGhG:87}.

As a consequence of~\eqref{gTGhG:87}, we can slide the ball~$
B_{R/2} (-2R,0,\dots,0,t)$
in direction~$e_1$ till it touches~$\partial E$. Notice that if no touching occurs
for any~$t$, then~\eqref{TG:1} holds true and we are done.
So we assume, by contradiction, that there exists~$t
\ge2M_{5R} +2R$ for which a touching occurs, namely
there exists a ball~$B:=B_{R/2} (\rho,0,\dots,0,t)$ for some~$\rho\in[ -2R,2R]$
such that
\begin{equation}\label{POgy5678}
B\subset E^c\end{equation} and there exists~$p\in (\partial B)\cap(\partial E)\cap\overline\Omega$.

Let now~$B'$ be the ball symmetric to~$B$ with respect to~$p$,
and let~$K$ be the convex envelope of~$B\cup B'$.

Notice that if~$x\in B'$ then~$x_n \ge t-\frac{3}{2}R\ge2M_{5R} +\frac{R}{2}
>2M_{5R}$. That is, $B\cup B'\subseteq \{x_n >2M_{5R}\}$ and so, by convexity
\begin{equation}\label{hj78HaJ}
K\subseteq \{x_n >2M_{5R}\}.\end{equation}
Now we claim that
\begin{equation}\label{BIS-hj78HaJ}
K\subseteq \{x_n > v(x')\}.\end{equation}
Indeed, if~$x\in K$ then~$|x'|\le \rho+2R\le 4R$, hence~\eqref{BIS-hj78HaJ}
follows from~\eqref{hj78HaJ}.

{F}rom \eqref{jnon8a9}
and~\eqref{BIS-hj78HaJ} we conclude that
\begin{equation}\label{TRIS-hj78HaJ}
K\setminus\Omega \subseteq E^c.\end{equation}
Now define~$B_\star:= B_1\big(p+(2R_o+2)e_1\big)$ and
we observe that
\begin{equation}\label{QUATRIS-hj78HaJ}
B_\star \subseteq \Omega^c.\end{equation}
Indeed, if~$x\in B_\star$, then
\begin{eqnarray*}
&& |x'|\ge \big| \big(p'+(2R_o+2)e_1\big)\big|-
\big|x' - \big(p'+(2R_o+2)e_1\big)\big|\\ &&\qquad\ge
2R_o+2 -|p'| -\big|x- \big(p+(2R_o+2)e_1\big)\big|
\ge 2R_o+2-R_o -1>R_o,\end{eqnarray*}
which proves~\eqref{QUATRIS-hj78HaJ}.

Now we check that
\begin{equation}\label{QUINRIS-hj78HaJ}
B_\star \subseteq K.\end{equation}
Indeed, 
\begin{equation}\label{QU8UFGhh}
{\mbox{if~$x\in B_\star$, then~$|x-p|\le 2R_o+3$,}}
\end{equation}
and so in particular~$|x-p|
<\frac{R}{4}$
if~$R$ is large enough, and this proves~\eqref{QUINRIS-hj78HaJ}.

In light of~\eqref{TRIS-hj78HaJ}, \eqref{QUATRIS-hj78HaJ}
and~\eqref{QUINRIS-hj78HaJ}, we have that
\begin{equation}\label{SEX-QUINRIS-hj78HaJ}
B_\star \subseteq K\cap\Omega^c=K\setminus\Omega\subseteq E^c.\end{equation}
Also, since we have slided the balls from left to right, we have that~$B_\star$
is on the right of~$B$ and hence it lies outside~$B$.
Hence, \eqref{QUINRIS-hj78HaJ}
can be precised by saying that~$B_\star\subseteq K\setminus B$. 

Thus, as a consequence of \eqref{POgy5678} and \eqref{SEX-QUINRIS-hj78HaJ},
\begin{eqnarray*}
&& \int_{K} \frac{\chi_{E^c}(y)- \chi_{E}(y)}{|p-y|^{n+2s}}\,dy
=\int_{K\setminus B_\star} \frac{\chi_{E^c}(y)- \chi_{E}(y)}{|p-y|^{n+2s}}\,dy
+\int_{B_\star} \frac{dy}{|p-y|^{n+2s}}
\\ &&\qquad\ge
\int_{B} \frac{dy}{|p-y|^{n+2s}}-
\int_{K\setminus(B\cup B_\star)} \frac{dy}{|p-y|^{n+2s}}
+\int_{B_\star} \frac{dy}{|p-y|^{n+2s}} 
\\ &&\qquad\ge
\int_{B} \frac{dy}{|p-y|^{n+2s}}-   
\int_{K\setminus B} \frac{dy}{|p-y|^{n+2s}}
+\int_{B_\star} \frac{dy}{|p-y|^{n+2s}},
\end{eqnarray*}
in the principal value sense. Hence, the contributions in~$B$ and~$B'$
cancel out by symmetry and, in virtue of Lemma~\ref{Trap}
(used here with~$\lambda:=1$),
we obtain that
$$ \int_{K} \frac{\chi_{E^c}(y)- \chi_{E}(y)}{|p-y|^{n+2s}}\,dy
\ge -CR^{-2s} +\int_{B_\star} \frac{dy}{|p-y|^{n+2s}},$$
up to renaming~$C>0$.
Now if~$y\in B_\star$ we have that~$|p-y|\le 2R_o+3\le C$,
for some~$C>0$, thanks to~\eqref{QU8UFGhh}. Also,
if~$y\in \R^n\setminus K$ then~$|p-y|\ge R/4$.
As a consequence, up to renaming~$C>c>0$ step by step,
\begin{eqnarray*}
&& \int_{\R^n} \frac{\chi_{E^c}(y)- \chi_{E}(y)}{|p-y|^{n+2s}}\,dy\ge
\int_{K} \frac{\chi_{E^c}(y)- \chi_{E}(y)}{|p-y|^{n+2s}}\,dy -CR^{-2s}
\\ &&\qquad \ge
-CR^{-2s} +\int_{B_\star} \frac{dy}{|p-y|^{n+2s}}\ge -CR^{-2s}+
c\,|B_\star|\ge -CR^{-2s} +c,
\end{eqnarray*}
which is strictly positive if~$R$ is large enough.
This is in contradiction with the 
Euler-Lagrange equation in the viscosity sense
(see Theorem~5.1 in~\cite{CRS}) and so it proves~\eqref{TG:1}.
\end{proof}

Next result gives the continuity of the fractional mean
curvature at the smooth points of the boundary:

\begin{lemma}\label{CONT:lemma:curvature}
Let
\begin{equation}\label{HG:ALP}
\alpha\in (2s,1].\end{equation}
Let~$E\subseteq\R^n$ and~$x_o\in \partial E$.
Assume that~$(\partial E)\cap B_R(x_o)$ is of class~$C^{1,\alpha}$,
for some~$R>0$. Then
$$ \lim_{{x\to x_o}\atop{x\in\partial E}}
\int_{\R^n} \frac{\chi_{E^c}(y)- \chi_{E}(y)}{|x-y|^{n+2s}}\,dy
=\int_{\R^n} \frac{\chi_{E^c}(y)- \chi_{E}(y)}{|x_o-y|^{n+2s}}\,dy.$$
\end{lemma}

\begin{proof} Up to a rigid motion, we suppose that~$x_o=0$
and that, in the vicinity of the origin, the set~$E$
is the subgraph of a function~$u\in C^{1,\alpha}(\R^{n-1})$
with~$u(0)=0$ and~$\nabla u(0)=0$. By formulas~(49) and~(50)
in~\cite{BEGO}, we can write the 
fractional mean
curvature in terms of~$u$, as long as~$|x'|$ is small enough.
More precisely, there exist an odd and smooth functions~$F$,
with~$F(0)=0$, $|F|+|F'|\le C$, for some~$C>0$,
a function~$\Psi\in C^{1,\alpha}(\R^{n-1})$,
and a smooth, radial and compactly supported function~$\zeta$ such that,
if~$|x'|$ is small and~$x_n=u(x')$, 
$$ \int_{\R^n} \frac{\chi_{E^c}(y)- \chi_{E}(y)}{|x-y|^{n+2s}}\,dy
=\int_{\R^{n-1}} F\left( \frac{u(x'+y')-u(x')}{|y'|}\right)\,
\frac{\zeta(y')}{|y'|^{n-1+2s}}\,dy'+\Psi(x'),$$
in the principal value sense.
Since also, by symmetry,
$$ \int_{\R^{n-1}} F\left( \frac{\nabla u(x')\cdot y'}{|y'|}\right)\,
\frac{\zeta(y')}{|y'|^{n-1+2s}}\,dy'=0$$
in the principal value sense, we write
\begin{equation}\label{KJFR589}
\int_{\R^n} \frac{\chi_{E^c}(y)- \chi_{E}(y)}{|x-y|^{n+2s}}\,dy
=\int_{\R^{n-1}} \left[
F\left( \frac{u(x'+y')-u(x')}{|y'|}\right)
-F\left( \frac{\nabla u(x')\cdot y'}{|y'|}\right)
\right]\,
\frac{\zeta(y')}{|y'|^{n-1+2s}}\,dy'+\Psi(x').\end{equation}
So we define
$$ G(x',y'):=
\left[F\left( \frac{u(x'+y')-u(x')}{|y'|}\right)
-F\left( \frac{\nabla u(x')\cdot y'}{|y'|}\right)\right]
\,
\frac{\zeta(y')}{|y'|^{n-1+2s}} .$$
Notice that
$$ \lim_{x'\to 0} G(x',y') = G(0,y').$$
Also, for any small~$|x'|$ and bounded~$|y'|$,
$$ \left| F\left( \frac{u(x'+y')-u(x')}{|y'|}\right)
-F\left( \frac{\nabla u(x')\cdot y'}{|y'|}\right)\right|\le
C\,\frac{|u(x'+y')-u(x')
-\nabla u(x')\cdot y'|}{|y'|} \le C\,|y'|^\alpha.$$
Therefore
$$ |G(x',y')|\le \frac{C}{|y'|^{n-1-\alpha+2s}}
\in L^1_{{\rm loc}}(\R^{n-1}),$$
thanks to~\eqref{HG:ALP}.
Accordingly, by the Dominated Convergence Theorem,
$$ \lim_{x'\to 0} \int_{\R^{n-1}} G(x',y')\,dy'=
\int_{\R^{n-1}} G(0,y')\,dy'.$$
Consequently,
\begin{eqnarray*}&& \lim_{x'\to 0}\int_{\R^{n-1}} \left[
F\left( \frac{u(x'+y')-u(x')}{|y'|}\right)
-F\left( \frac{\nabla u(x')\cdot y'}{|y'|}\right)
\right]\,
\frac{\zeta(y')}{|y'|^{n-1+2s}}\,dy'+\Psi(x')
\\&&\qquad=
\int_{\R^{n-1}} \left[
F\left( \frac{u(y')-u(0)}{|y'|}\right)
-F\left( \frac{\nabla u(0)\cdot y'}{|y'|}\right)
\right]\,
\frac{\zeta(y')}{|y'|^{n-1+2s}}\,dy'+\Psi(x'),\end{eqnarray*}
which, combined with~\eqref{KJFR589}, establishes the desired result.
\end{proof}

The result in Lemma~\ref{CONT:lemma:curvature} can be modified
to take into account sets with lower regularity properties.

\begin{lemma}\label{PAL7}
Let~$R>0$, $E\subseteq\R^n$ and~$x_o\in \partial E$.
For any $k\in\N$, let~$x_k\in\partial E$, with~$x_k\to x_o$
as~$k\to+\infty$, be such that~$E$ is
touched from the inside at~$x_k$
by a ball of radius~$R$, i.e. there exists~$p_k\in \R^n$ such that
\begin{equation}\label{POKJGH=k}
B_R(p_k)\subseteq E\end{equation}
and~$x_k\in \partial B_R(p_k)$.

Suppose that
$$ \int_{\R^n} \frac{\chi_{E}(y)- \chi_{E^c}(y)}{|
x_k-y|^{n+2s}}\,dy\le0.$$
Then
\begin{equation}\label{sense}
\int_{\R^n} \frac{\chi_{E}(y)- \chi_{E^c}(y)}{|
x_o-y|^{n+2s}}\,dy\le0.\end{equation}
\end{lemma}

\begin{proof} 
Fix~$\lambda>0$, to be taken arbitrarily small
in the sequel. Let~$
q_k:=p_k+2(x_k-p_k)$. We observe that the ball~$B_R(q_k)$
is tangent to~$B_R(p_k)$ at~$x_k$.
Therefore, by Lemma~\ref{Trap},
\begin{equation}\label{JU:P} 
\int_{B_\lambda(x_k)\setminus
\big(B_R(p_k)\cup B_R(q_k)\big)} \frac{dy}{|x_k-y|^{n+2s}}
\le CR^{-2s} \lambda^{1-2s},\end{equation}
for some~$C>0$.
Also, using~\eqref{POKJGH=k},
\begin{equation}
\label{KH67-2}
\int_{B_\lambda(x_k)}
\frac{\chi_{E}(y)- \chi_{E^c}(y)}{|x_k-y|^{n+2s}}\,dy
\ge
\int_{B_\lambda(x_k)}
\frac{\chi_{B_R(p_k)}(y)- \chi_{B_R^c(p_k)}(y)}{|x_k-y|^{n+2s}}\,dy.\end{equation}
Now we define~$T_k$ to be the half-space passing through~$x_k$
with normal parallel to~$x_k-p_k$ and containing~$B_R(p_k)$.
By symmetry,
$$ \int_{B_\lambda(x_k)}
\frac{\chi_{T_k}(y)- \chi_{T_k^c}(y)}{|x_k-y|^{n+2s}}\,dy=0.$$
Using this, \eqref{KH67-2} and~\eqref{JU:P}, we obtain that
\begin{equation}\label{ALL-1}
\begin{split}
& \int_{B_\lambda(x_k)}
\frac{\chi_{E}(y)- \chi_{E^c}(y)}{|x_k-y|^{n+2s}}\,dy
\\&\qquad\ge
\int_{B_\lambda(x_k)}
\frac{\chi_{B_R(p_k)}(y)- \chi_{B_R^c(p_k)}(y)}{|x_k-y|^{n+2s}}\,dy
-
\int_{B_\lambda(x_k)}
\frac{\chi_{T_k}(y)- \chi_{T_k^c}(y)}{|x_k-y|^{n+2s}}\,dy
\\ &\qquad=- 2\int_{B_\lambda(x_k)\cap \big( T_k\setminus B_R(p_k)\big)}
\frac{dy}{|x_k-y|^{n+2s}}
\\ &\qquad\ge -CR^{-2s} \lambda^{1-2s}.
\end{split}\end{equation}
Now we define
$$ f_k(y):= \chi_{B_\lambda^c(x_k)}\cdot
\frac{\chi_{E}(y)- \chi_{E^c}(y)}{|x_k-y|^{n+2s}}.$$
We observe that~$f_k$ vanishes in~$B_\lambda(x_k)$. Also,
if~$y\in B_{2\lambda}(x_o)\setminus B_\lambda(x_k)$, we have that~$
|f_k(y)|\le \frac{1}{\lambda^{n+2s}}$. Moreover, if~$y\in
\R^n\setminus B_{2\lambda}(x_o)$, we have that
$$ |y-x_o|\le |y-x_k|+|x_k-x_o|\le |y-x_k|+\lambda\le|y-x_k|+
\frac{|y-x_o|}{2},$$
as long as~$k$ is large enough, and so~$|y-x_k|\ge \frac{|y-x_o|}{2}$,
which gives that~$|f_k(y)|\le  \frac{1}{|x-x_o|^{n+2s}}$
for any~$y\in\R^n\setminus B_{2\lambda}(x_o)$.
As a consequence of these observations, we can use the
Dominated Convergence Theorem and obtain that
$$ \lim_{k\to+\infty}
\int_{B_\lambda^c(x_k)}
\frac{\chi_{E}(y)- \chi_{E^c}(y)}{|x_k-y|^{n+2s}}\,dy
=
\lim_{k\to+\infty} \int_{\R^n} f_k(y)\,dy
= \int_{\R^n} 
\lim_{k\to+\infty} 
f_k(y)\,dy=
\int_{B_\lambda^c(x_o)}
\frac{\chi_{E}(y)- \chi_{E^c}(y)}{|x_o-y|^{n+2s}}\,dy.$$
Thus, if~$k$ is large enough,
$$ \int_{B_\lambda^c(x_k)}
\frac{\chi_{E}(y)- \chi_{E^c}(y)}{|x_k-y|^{n+2s}}\,dy \ge
\int_{B_\lambda^c(x_o)}
\frac{\chi_{E}(y)- \chi_{E^c}(y)}{|x_o-y|^{n+2s}}\,dy-
R^{-2s} \lambda^{1-2s}.$$
Thus, recalling~\eqref{ALL-1},
\begin{eqnarray*}
0&\ge&
\int_{\R^n}
\frac{\chi_{E}(y)- \chi_{E^c}(y)}{|x_k-y|^{n+2s}}\,dy 
\\ &=&
\int_{B_\lambda(x_k)}
\frac{\chi_{E}(y)- \chi_{E^c}(y)}{|x_k-y|^{n+2s}}\,dy 
+
\int_{B_\lambda^c(x_k)}
\frac{\chi_{E}(y)- \chi_{E^c}(y)}{|x_k-y|^{n+2s}}\,dy 
\\ &\ge&
\int_{B_\lambda^c(x_o)}
\frac{\chi_{E}(y)- \chi_{E^c}(y)}{|x_o-y|^{n+2s}}\,dy-
CR^{-2s} \lambda^{1-2s},\end{eqnarray*}
up to renaming~$C>0$ line after line. Then, \eqref{sense},
in the principal value sense, follows
by sending~$\lambda\to0$.
\end{proof}

A variation of Lemma~\ref{PAL7}
deals with
the touching by sufficiently smooth hypersurfaces, instead of balls.
In this sense, the result needed 
for our scope 
is the following:

\begin{lemma}\label{lemma36}
Let~$\Lambda>0$.
Let~$E\subseteq\R^n$ and~$x_o\in \partial E$.
For any $k\in\N$, let~$x_k\in\partial E$, with~$x_k\to x_o$ 
as~$k\to+\infty$, 
be such that~$E$
is
touched from the inside in~$B_\Lambda(x_k)$
at~$x_k$ by a surface
of class~$C^{1,\alpha}$, with $C^{1,\alpha}$-norm bounded
independently of~$k$ and~$\alpha\in(2s,1]$.
Suppose that
$$ \int_{\R^n} \frac{\chi_{E}(y)- \chi_{E^c}(y)}{|
x_k-y|^{n+2s}}\,dy\le0.$$
Then
\begin{equation*}
\int_{\R^n} \frac{\chi_{E}(y)- \chi_{E^c}(y)}{|
x_o-y|^{n+2s}}\,dy\le0.\end{equation*}
\end{lemma}

\begin{proof} The proof is similar to the one of Lemma~\ref{PAL7}.
The only difference is that~\eqref{JU:P} is replaced here by
\begin{equation}\label{JU:P:here}
\int_{B_\lambda(x_k)\setminus (P_k^+\cup P_k^-)} \frac{dy}{|x_k-y|^{n+2s}}
\le C\,\lambda^{\alpha-2s },\end{equation}
where~$\lambda\in(0,\Lambda)$ can be taken arbitrarily small
and~$P_k^+$ is a region with $C^{1,\alpha}$-boundary
that is contained in~$E$ and~$P_k^-$ is the even reflection of~$P_k^+$
with respect to the tangent plane of~$P_k^+$ at~$x_k$.
In this framework, \eqref{JU:P:here}
is a consequence of
Lemma~\ref{TRAP-C1}.

The rest of the proof follows the arguments given in
the proof of
Lemma~\ref{PAL7}, substituting~$B_R(p_k)$ and~$B_R(q_k)$
with~$P_k^+$ and~$P_k^-$.
\end{proof}

\section{Graph properties of $s$-minimal sets and
proof of Theorem \ref{GRAPH}}\label{HY:08}

The goal of this section is to prove Theorem~\ref{GRAPH}.

\begin{proof}[Proof of Theorem~\ref{GRAPH}]
First we show that~$(\partial E)\cap\Omega$ is a graph, namely that
\begin{equation}\label{67-09uihfGG77}
{\mbox{formula \eqref{p98T78-00gra}
holds true, for some~$v:\R^{n-1}\to\R$.}}
\end{equation}
The idea is to slide $E$ from above till it touches itself.
Namely, for any~$t\ge0$, we let~$E_t:= E+te_n$.
We also define
$$ \Gamma:=\partial\Big( (\partial E) \cap \Omega\Big) \;{\mbox{ and }}\;
\Gamma_t:=\partial\Big( (\partial E_t) \cap \Omega\Big).$$
By Lemma~\ref{1212},
\begin{equation}\label{HJ:GRAPH:1}
\Omega_o\times(-\infty,-M)\,\subseteq\, E\cap \Omega\,
\subseteq\,\Omega_o\times(-\infty,M),\end{equation}
for some~$M\ge0$.
Hence, if~$t>2M$, then~$\Gamma_t$ lies above~$\Gamma$
(with respect to the $n$th coordinate).
So we take the smallest~$t$ for which such position holds, namely
we set
\begin{equation}\label{DEt}
t := \inf \{ \tau {\mbox{ s.t. $\Gamma_\tau$ lies above~$\Gamma$}} \}.\end{equation}
Our goal is to show that
\begin{equation}\label{t=0}
t=0.\end{equation}
Indeed, if we show that~$t=0$, 
we could define~$v(x'):= \inf\{ \tau {\mbox{ s.t. }}
(x',\tau)\in E^c\}$ and obtain that~$E\cap \Omega_o$
is the subgraph of~$v$.

To prove that~$t=0$, we argue by contradiction, assuming that
\begin{equation}\label{11K45fQQ}
t>0,\end{equation}
and so there is a contact point between $\Gamma$
and~$\Gamma_t$.

We remark that, in our framework, the set $\partial E$
may have some vertical portions along $\partial\Omega$
(and indeed, this is the ``typical'' picture that we deal with,
see \cite{NOSTRO}). Hence, the two sets~$\partial E$
and~$\partial E_t$ may share some common vertical portions along~$\partial \Omega$.
Roughly speaking, these vertical portions do not really consist
of contact points since they do not prevent the sliding of the sets
$E$ and $E_t$ by keeping the inclusion (for instance,
in Figure \ref{ES098-1} the only contact
point is the black dot named $p$, while in Figure \ref{ES098-2}
we have two contact points given by the black dots $p$ and $q$).

\begin{figure}
    \centering
    \includegraphics[width=14cm]{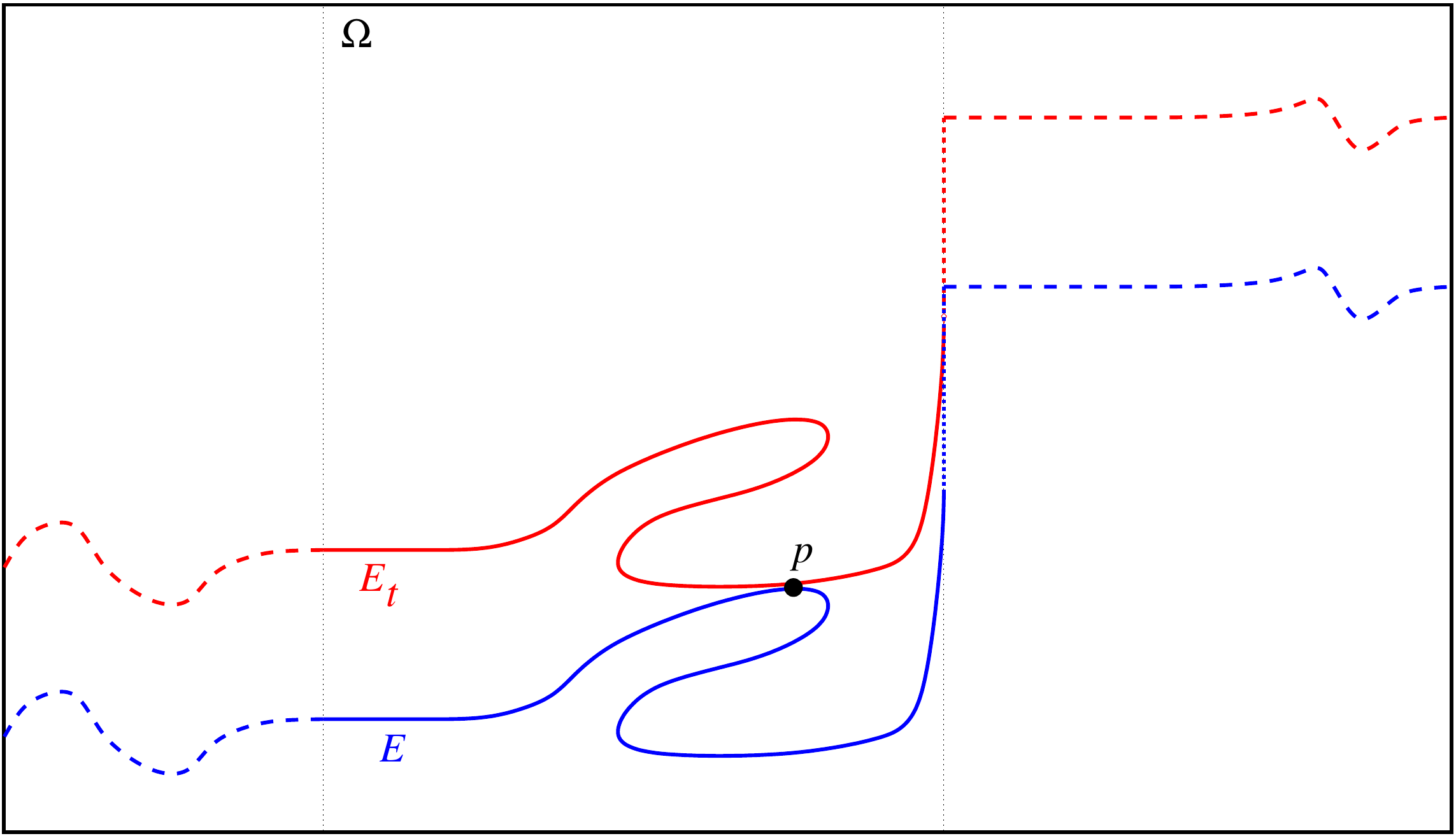}
    \caption{The case in \eqref{POSS-GRAPH:1}.}
    \label{ES098-1}
\end{figure}

To formalize this notion, we explicitly define the set of contact
points between $\Gamma$ and $\Gamma_t$ as
\begin{equation}\label{KDef10}
{\mathcal{K}} := \Gamma\cap\Gamma_t=
\Big(\partial\big( (\partial E) \cap\Omega\big)\Big) \cap 
\Big(\partial\big((\partial E_t) \cap\Omega\big)\Big).\end{equation}
The definition of first contact time given in \eqref{DEt}
gives that ${\mathcal{K}}\ne \varnothing$.

We distinguish two cases, according to whether all the contact
points are interior, or there are boundary contacts
(no other possibilities occur, thanks to~\eqref{HJ:GRAPH:2}). Namely,
we have that
either 
\begin{equation}\label{POSS-GRAPH:1}
{\mbox{all the contact points lie in~$\Omega_o\times\R$,}}
\end{equation}
i.e. ${\mathcal{K}} \subseteq \Omega$,
or
\begin{equation}\label{POSS-GRAPH:2}
{\mbox{there exists a contact point in~$(\partial\Omega_o)\times\R$,}}
\end{equation}
i.e. ${\mathcal{K}} \cap(\partial \Omega)\ne\varnothing$.

\begin{figure}
    \centering
    \includegraphics[width=14cm]{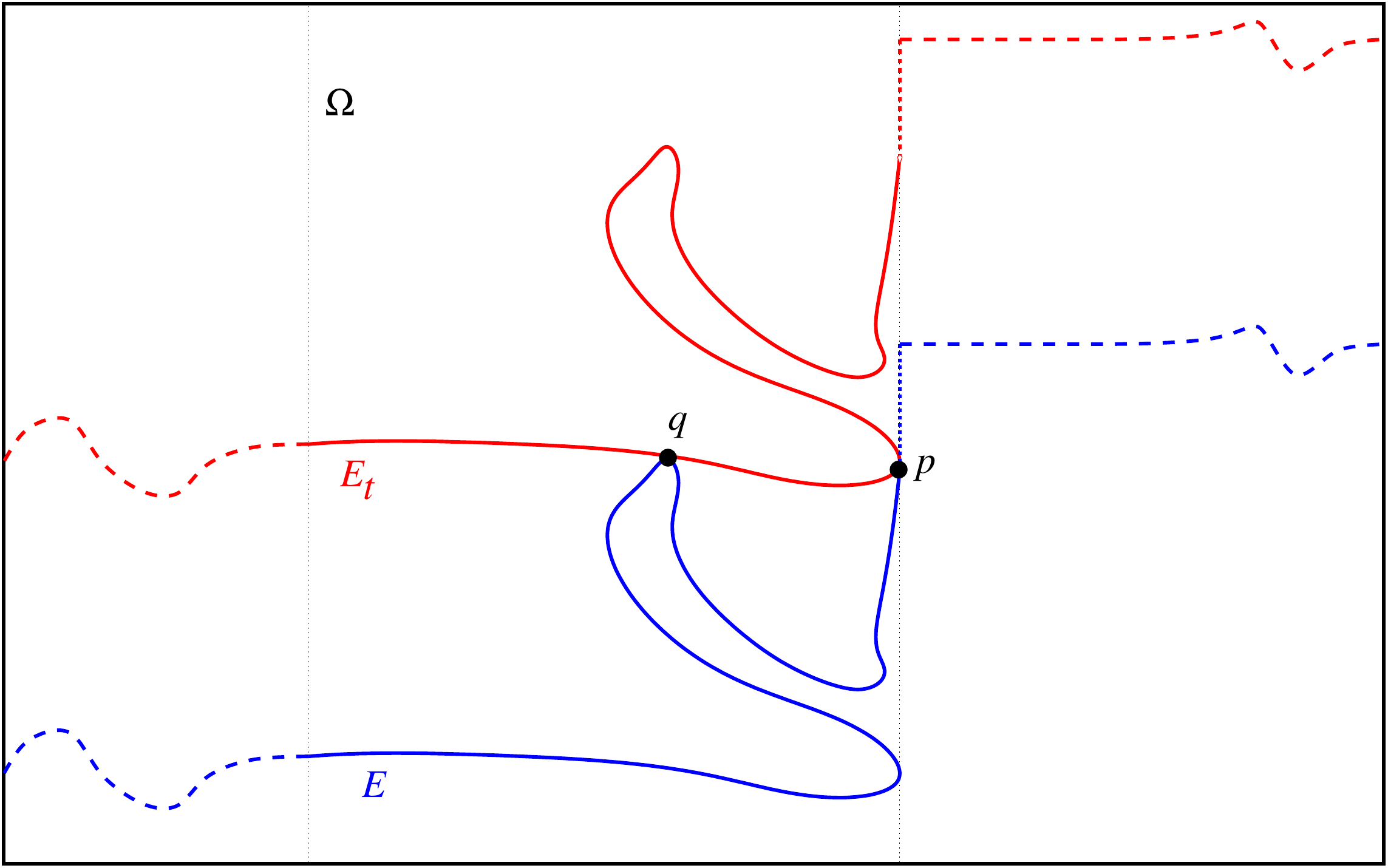}
    \caption{The case in \eqref{POSS-GRAPH:2}.}
    \label{ES098-2}
\end{figure}

The case in \eqref{POSS-GRAPH:1} is depicted in Figure \ref{ES098-1}
and the case in
\eqref{POSS-GRAPH:2} is depicted in Figure \ref{ES098-2}.
The rest of the proof will take into account these two cases
separately.
To be precise,
when~\eqref{POSS-GRAPH:1} holds true,
we will get a contradiction by using the supconvolution method
(since we did not assume any regularity of the surface that we
study),
while if \eqref{POSS-GRAPH:2} holds true we will
take advantage of the
regularity in the obstacle problem for nonlocal
minimal surfaces (see~\cite{PRO} and Theorem~\ref{PRO:TH} here).
\medskip

\noindent{\em The case in which~\eqref{POSS-GRAPH:1} holds true.}
Assume first that~\eqref{POSS-GRAPH:1} is satisfied.
Then we consider the subconvolution of~$E$
and we slide it from above till it touches the supconvolution of~$E$
(in the notation of Section~\ref{SEC:SUBCO}).
More explicitly, fixed~$\delta$ and $\eta>0$, to be taken suitably
small in the sequel,
and,
for any~$\tau\in\R$, we consider~$E^\flat_\delta+\tau e_n$.

Now we recall the notation in \eqref{Oema}
and,
by taking $\eta$ and~$\delta$ sufficiently small, we deduce from~\eqref{POSS-GRAPH:1}
that ${\mathcal{K}}$ lies inside $\Omega_{16(\delta+\eta)}$, at some
positive distance from $\partial\Omega$ (which is uniform in $\eta$
and~$\delta$).

By~\eqref{HJ:GRAPH:1},
we have that if~$\tau$ is large, then
$$ (E^\flat_\delta+\tau e_n)\cap
\Omega\supseteq
E^\sharp_\delta\cap \Omega$$
and so, in particular,
\begin{equation}\label{1x2}
(E^\flat_\delta+\tau e_n)\cap
\Omega_\eta\supseteq
E^\sharp_\delta\cap \Omega_\eta.\end{equation}
So we take the smallest~$\tau=\tau_{\delta,\eta}$ for which
such inclusion holds. {F}rom~\eqref{11K45fQQ}, we have that
\begin{equation}\label{ZR56}
\tau\ge \frac{t}{2}>0,
\end{equation}
for small~$\delta$ and $\eta$.
Also, by~\eqref{POSS-GRAPH:1}
(recall also the first statement in Corollary~\ref{Cor-CONV}),
if~$\delta$ is small enough,
we obtain that~$\big(\partial(E^\flat_\delta+\tau e_n)\big)\cap
\Omega_\eta$
and~$(\partial E^\sharp_\delta)\cap
\Omega_\eta$ possess a contact point~$p$
in~$\Omega_o\times\R$ (namely, $p$ is close to the contact set~${\mathcal{K}}$
for small $\delta$ and $\eta$).
Now we distinguish two subcases:
either this is the first contact point in the whole of the space
or not. In the first subcase, we have that \eqref{1x2}
may be strengthen to~$E^\flat_\delta+\tau e_n
\supseteq
E^\sharp_\delta$, and therefore we can apply
Proposition~\ref{TOUCH:omega}, and we obtain
that~$E^\sharp_\delta=E^\flat_\delta+\tau e_n$.
By taking~$\delta$ arbitrarily small
and using~\eqref{ZR56}, we obtain that~$E=E+\tau_o e_n$,
with~$\tau_o\ge t/2>0$, which is in contradiction with~\eqref{HJ:GRAPH:2}.

The second subcase is when
the first contact point~$p$ in~$\Omega_\eta$
does not prevent the sets to overlap outside~$\Omega_\eta$.
In this case, we will show that this overlap only occurs either
in $\Omega\setminus\Omega_\eta$ or
at
infinity, and then we provide a contradiction arising from
the contribution in bounded sets.
Namely, first of all we 
recall the notation in~\eqref{CYLINDER:DEF} and \eqref{Oema2}
and we
notice that for any~$R>0$ there exist~$\delta_R$, $\eta_R>0$
such that for any~$\delta\in(0,\delta_R]$ and $\eta\in(0,\eta_R]$
we have that
\begin{equation}\label{inclusion}
(E^\flat_\delta+\tau e_n)\cap 
{\mathcal{D}}_{R,\eta}
\supseteq
E^\sharp_\delta\cap
{\mathcal{D}}_{R,\eta}.
\end{equation}
To prove~\eqref{inclusion}, we argue by contradiction.
If not, there exist some~$R>0$ and infinitesimal sequences~$\delta$, $\eta\to0$
such that~$ 
\big(E^\sharp_\delta
\setminus 
(E^\flat_\delta+\tau e_n)
\big)\cap 
{\mathcal{D}}_{R,\eta}\ne\varnothing$. Then, let~$q_{\delta,\eta}
=(q_{\delta,\eta}',q_{\delta,\eta,n})$ be a point in such set.
By construction~$|q_{\delta,\eta,n}|\le 3M+1$ and~$|q_{\delta,\eta}'|\le R$,
therefore, up to subsequences, as~$\delta$, $\eta\to0$,
we may suppose that~$\tau=\tau_{\delta,\eta}
\to\tau_\star$ and~$q_{\delta,\eta}\to q_\star=(q_\star',q_{\star,n})\in
\overline{
\big(E \setminus 
(E+\tau_\star e_n)\big)\cap
{\mathcal{D}}_{R,\eta} }$.
Hence, by~\eqref{1x2}, $q_\star\in\R^n\setminus\Omega$
and so, by~\eqref{HJ:GRAPH:2},
we have that~$u(q_\star')+\tau_\star\le q_{\star,n}\le u(q_\star')$.
This gives that~$\tau_\star\le0$, which is in contradiction with~\eqref{ZR56}
and thus completes the proof of~\eqref{inclusion}.

Now we fix~$R_o>0$ such that~$\Omega\subset {\mathcal{C}}_{R_o}$,
and we suppose that~$R>4(R_o+1)$.
Thanks to~\eqref{inclusion}, we can now use Proposition~\ref{TOUCH:omega-bis}
and obtain that
\begin{equation}\label{11K45fQQ-BIS}
\int_{{\mathcal{D}}_{R,\eta}} \frac{ \chi_{(E^\flat_\delta+\tau e_n)
\setminus E^\sharp_\delta}(y) }{|p-y|^{n+2s}}\,dy=
\int_{{\mathcal{D}}_{R,\eta}} \frac{\chi_{(E^\flat_\delta+\tau e_n)
\setminus E^\sharp_\delta}(y)-
\chi_{E^\sharp_\delta\setminus(E^\flat_\delta+\tau e_n)
}(y)}{|p-y|^{n+2s}}\,dy\le C(R^{-2s}+\eta),
\end{equation}
for some~$C>0$ that does not depend on
$R$, $\delta$ and $\eta$, provided that~$\delta$ and $\eta$
are small enough.

Since~$u$ is continuous in~$\R^{n-1}$, it is uniformly continuous
in compact sets and so we can define
$$ \sigma_\delta := \sup_{{|x'|, \;|y'|\le R_o +3 }\atop{|x'-y'|\le2\delta}}
|u(x')-u(y')|,$$
and we have that~$\sigma_\delta\to0$ as~$\delta\to0$.

We claim that, for small~$\delta>0$,
\begin{equation}\label{UNI}
\begin{split}
&{\mbox{if~$x=(x',x_n)\in \partial(E^\flat_\delta+\tau e_n)$,
$y=(y',y_n)\in \partial E^\sharp_\delta$ and~$x'=y'$,}}\\
&{\mbox{with~$|x'|\in (R_o+1,R_o+2)$,}}\\
&{\mbox{then $x_n \ge y_n+\frac{t}{4} $.}}
\end{split}
\end{equation}
To prove it,
we use the
first statement in Corollary~\ref{Cor-CONV} to find~$x_o\in (\partial E)+\tau e_n$
and~$y_o\in \partial E$ such that
$$\max\{|x-x_o|, \,|y-y_o|\}\le\delta.$$
Notice that~$x_{o,n}=u(x_o')+\tau$ and~$y_{o,n}=u(y_o')$.
Moreover, $|x'-x_o'|\le\delta$ and~$|x'-y_o'|=|y'-y_o'|\le\delta$,
hence~$|x'_o-y'_o|\le2\delta$.
Therefore
\begin{eqnarray*}
&& x_n - y_n =
x_n -x_{o,n} +u(x_o')+\tau - y_n +y_{o,n}-u(y_o') \\
&&\qquad\ge \tau -|x-x_{o}|-|y-y_o| - |u(x_o')-u(y_o')|\ge
\tau-2\delta -\sigma_\delta.
\end{eqnarray*}
This and~\eqref{ZR56}
imply~\eqref{UNI}, as desired.

Notice also that ${\mathcal{D}}_{R,\eta}\supset
{\mathcal{C}}_{R_o+2}\setminus {\mathcal{C}}_{R_o+1}$.
So we use~\eqref{inclusion} and~\eqref{UNI} to deduce that,
fixed~$R>4(R_o+1)$ and~$\delta>0$ small enough (possibly in dependence
of~$R$), 
$$ \int_{{\mathcal{D}}_{R,\eta} } \frac{\chi_{(E^\flat_\delta+\tau e_n)
\setminus E^\sharp_\delta} (y)}{|p-y|^{n+2s}}\,dy\ge
\int_{{\mathcal{C}}_{R_o+2}\setminus {\mathcal{C}}_{R_o+1}}
\frac{\chi_{(E^\flat_\delta+\tau e_n)
\setminus E^\sharp_\delta}(y)}{|p-y|^{n+2s}}\,dy\ge c_o t,$$
for some~$c_o>0$ (possibly depending on the fixed~$R_o$ and~$M$).
{F}rom this and~\eqref{11K45fQQ-BIS}, we obtain that~$t\le \tilde C( R^{-2s}+\eta)$,
for some~$\tilde C>0$ and so, by taking~$\eta$ as small as
we wish and $R$ as large as we wish,
we conclude that~$t=0$. This is in contradiction with~\eqref{11K45fQQ},
and so
we have completed the proof of Theorem~\ref{GRAPH}
under assumption~\eqref{POSS-GRAPH:1}. \medskip

\noindent{\em The case in which~\eqref{POSS-GRAPH:2} holds true.}
Now we deal with the case in which~\eqref{POSS-GRAPH:2}
is satisfied. Hence, there exists
a contact point~$p=(p',p_n)\in (\partial E_t)\cap(\partial E)$
with~$p'\in\partial\Omega_o$. More explicitly, we
notice that, by \eqref{KDef10},
\begin{equation}\label{obe}
p\in \Big(\overline{ (\partial E_t)\cap\Omega }\Big)\cap\Big(
\overline{ (\partial E)\cap\Omega }\Big).
\end{equation}
Now, we observe that~$E$ is a variational subsolution in a neighborhood of~$p$
(according to Definition~2.3 in~\cite{CRS}): namely,
if~$A\subseteq E\cap \Omega$ and~$p\in\overline{A}$, we have that
$$ 0\ge \per(E,\Omega)-\per(E\setminus A,\Omega)
=L(A,E^c)-L(A,E\setminus A).$$
Therefore (see Theorem~5.1 in~\cite{CRS}) we have that
\begin{equation}\label{iert}
\int_{\R^n} \frac{\chi_{E}(y)-
\chi_{\R^n\setminus E}(y)}{|p-y|^{n+2s}}\,dy\ge0.
\end{equation}
in the viscosity sense (i.e.~\eqref{iert}
holds true provided that~$E$ is touched
by a ball from outside at~$p$).

Our goal is now to establish fractional mean curvature estimates
in the strong sense.
For this, we notice that,
by~\eqref{11K45fQQ},
either
\begin{equation}\label{Ppos-p-1}
p_n\ne u(p')\end{equation}
or
\begin{equation}\label{Ppos-p-2}
p_{n}\ne u(p')+t.\end{equation}
We focus on the case in which~\eqref{Ppos-p-1} holds true
(the case in~\eqref{Ppos-p-2} can be treated
similarly, by exchanging the roles of~$E$ and~$E_t$). 

Then,
either~$B_r(p)\setminus\Omega \subseteq E$
or~$B_r(p)\setminus\Omega \subseteq E^c$, for a small~$r>0$.
In any case, by Theorem~\ref{PRO:TH},
we have that~$(\partial E)\cap B_r(p)$
is a $C^{1,\frac{1}{2}+s}$-graph in the direction of
the normal of~$\Omega$ at~$p$, up to renaming~$r$.

Let~$\nu(p)=(\nu'(p),\,\nu_n(p))$ be such normal, say, in the interior
direction.
Since~$\Omega$ is a cylinder, we have that~$\nu_n(p)=0$. 
Also, up to a rotation
we can suppose that~$\nu'(p)=e_1$. In this framework,
we can write~$\partial E$ in the vicinity of~$p$
as a graph~$G:=\{ x_1=\Psi (x_2,\dots,x_n)\}$,
for a suitable~$\Psi\in C^{1,\frac{1}{2}+s}(\R^{n-1})$,
with~$\Psi(p_2,\dots,p_n)=p_1$.

We observe that
\begin{equation}\label{obe2}
{\mbox{there exists a sequence of points~$p^{(k)}\in G$ such that~$p^{(k)}
\in\Omega$ and~$p^{(k)}\to p$ as $k\to+\infty$.}}
\end{equation}
Indeed, if not, we would have that~$\partial E$ in the vicinity
of~$p$ lies in~$\Omega^c$. This is in contradiction with~\eqref{obe}
and so it proves~\eqref{obe2}.

{F}rom~\eqref{obe2}, we obtain that there exists a sequence
of points~$p^{(k)}\to p$, such that
\begin{equation}\label{Jkms33d}
{\mbox{$\partial E$ near~$p^{(k)}$
is a graph of class~$C^{1,\frac{1}{2}+s}$}}\end{equation}
and
$$ \int_{\R^n} \frac{\chi_{E}(y)-
\chi_{E^c}(y)}{|p^{(k)}-y|^{n+2s}}\,dy=0.$$
As a consequence of this, \eqref{Jkms33d}, and Lemma~\ref{lemma36} 
(applied to both~$E$ and~$E^c$)
we obtain that
\begin{equation*}
\int_{\R^n} \frac{\chi_{E}(y)-
\chi_{E^c}(y)}{|p-y|^{n+2s}}\,dy=0.
\end{equation*}
Hence, since~$E_t\supseteq E$ (and they are not equal,
thanks to~\eqref{11K45fQQ}),
\begin{equation}\label{Hg:i679HJl67}
\int_{\R^n} \frac{\chi_{E_t}(y)-
\chi_{E^c_t}(y)}{|p-y|^{n+2s}}\,dy>0.
\end{equation}
Also, since~$E_t\supseteq E$, we have that~$(\partial E_t)
\cap B_{\frac{r}{4}}(p)$
can only lie on one side of the graph~$G$,
i.e. 
\begin{eqnarray}
\label{ALT:er:1}
&&{\mbox{either~$E_t\cap B_{\frac{r}{4}}(p)\supseteq
\{ x_1\ge\Psi (x_2,\dots,x_n)\}$}}\\
\label{ALT:er:2}
&&{\mbox{or~$E_t\cap B_{ \frac{r}{4} }(p)\subseteq
\{ x_1\le\Psi (x_2,\dots,x_n)\}$,}}
\end{eqnarray}
see Figure \ref{h67FF}.

\begin{figure}
    \centering
    \includegraphics[height=5.9cm]{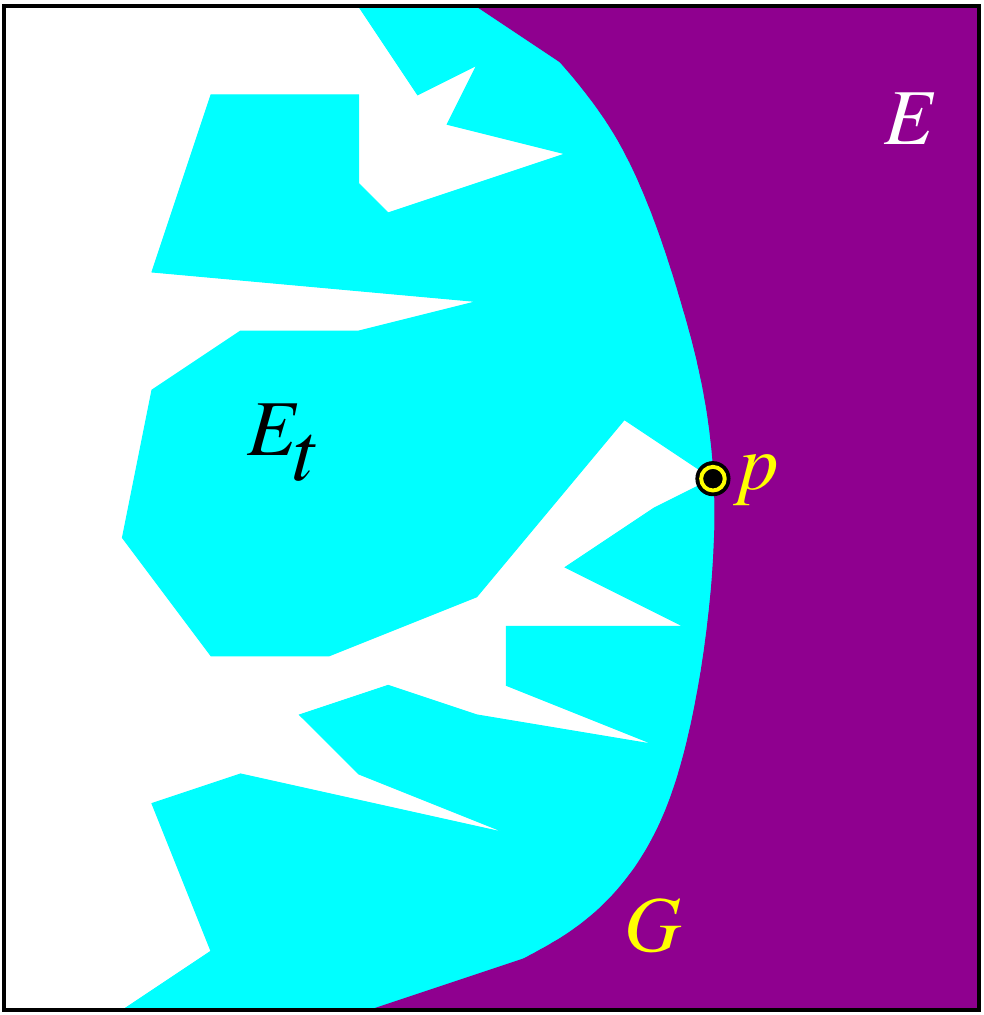}
$\;\;\;\;\;\;\;\;\;$
    \includegraphics[height=5.9cm]{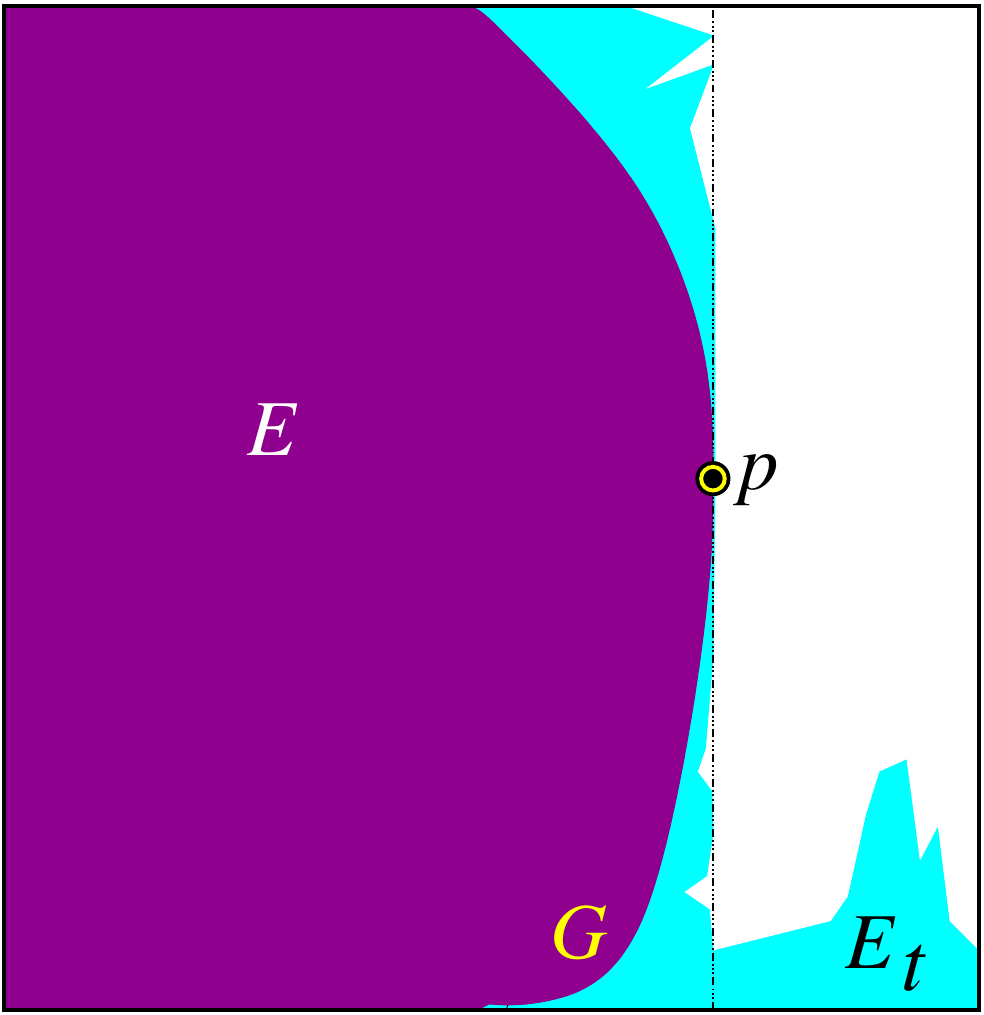}
    \caption{The alternative in \eqref{ALT:er:1} and~\eqref{ALT:er:2}.}
    \label{h67FF}
\end{figure}

In any case (recall~\eqref{obe}), we have that there exists a
sequence of points~$\tilde p^{(k)}\in(\partial E_t)\cap\Omega$
that can be touched by a surface of class
class~$C^{1,\frac{1}{2}+s}$ lying in~$E_t$
(indeed, for this we can either enlarge balls centered at~$G$,
or slide a translation of~$G$, see Figure~\ref{10-h67FF}).

\begin{figure}
    \centering
    \includegraphics[height=5.9cm]{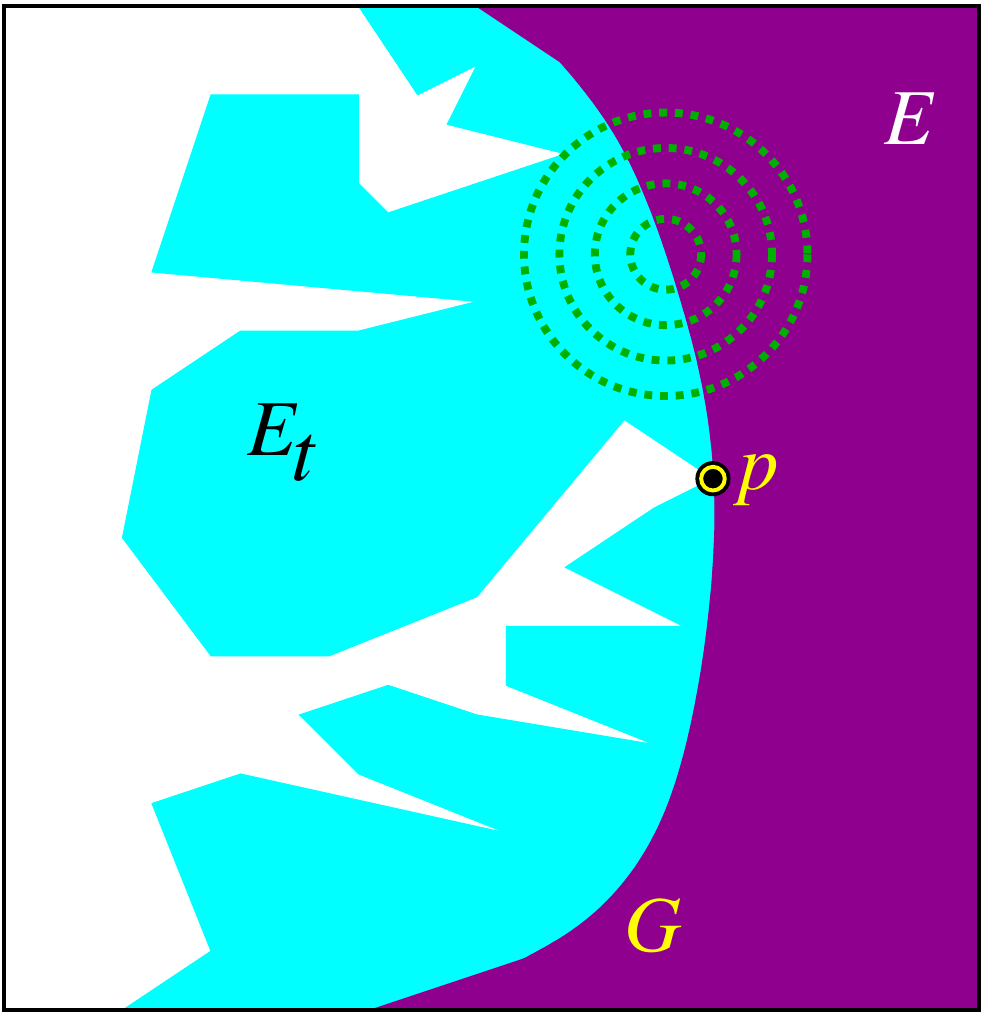}
$\;\;\;\;\;\;\;\;\;$
    \includegraphics[height=5.9cm]{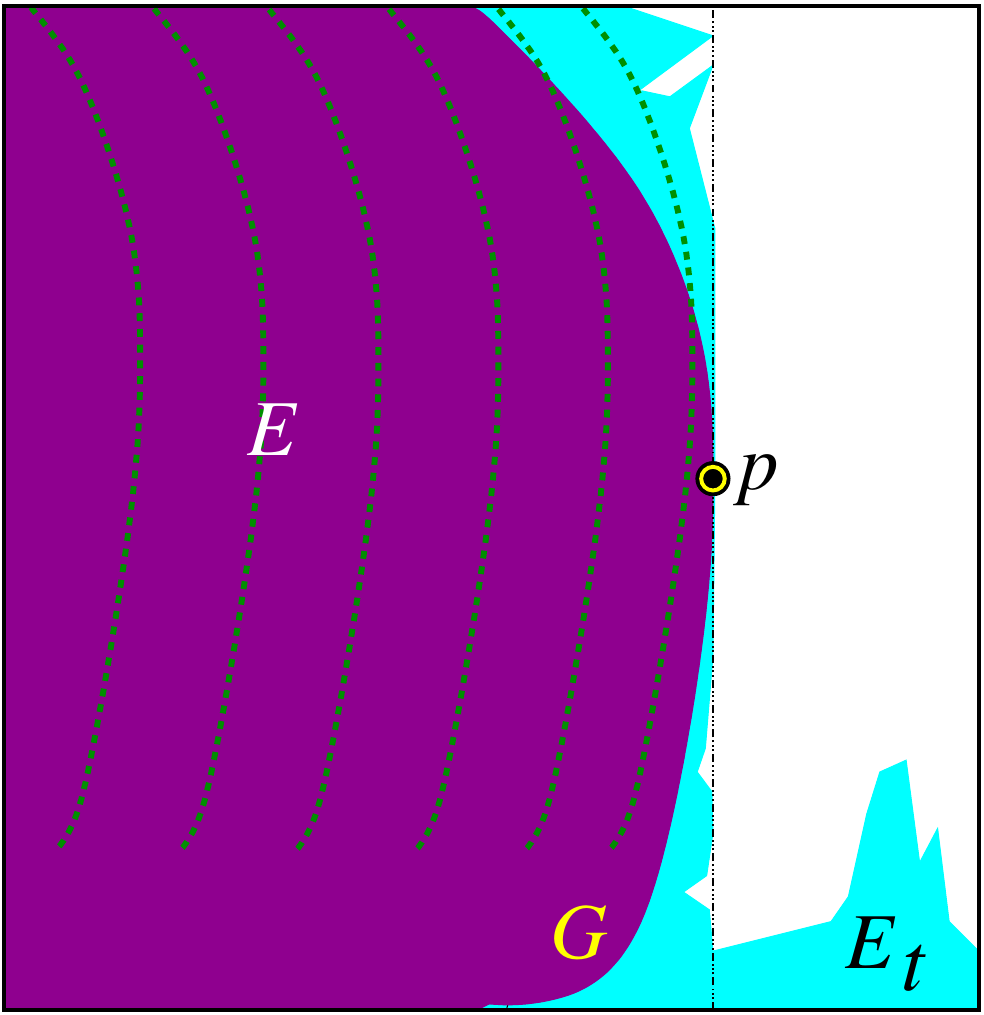}
    \caption{Touching $\partial E_t$,
according to the alternative in \eqref{ALT:er:1} and~\eqref{ALT:er:2}.}
    \label{10-h67FF}
\end{figure}

Then
$$ \int_{\R^n} \frac{\chi_{E_t}(y)-
\chi_{E^c_t}(y)}{|\tilde p^{(k)}-y|^{n+2s}}\,dy\leq0.$$
Hence, by Lemma \ref{lemma36},
$$ \int_{\R^n} \frac{\chi_{E_t}(y)-
\chi_{E^c_t}(y)}{|p-y|^{n+2s}}\,dy\leq0.$$
This is in contradiction with~\eqref{Hg:i679HJl67}
and so the proof of \eqref{67-09uihfGG77} is complete.

Now, to finish the proof of
Theorem~\ref{GRAPH}, we need to check the
properties on the function~$v$ stated in Theorem~\ref{GRAPH}.
By construction, $v$ and~$u$ coincide outside~$\Omega_o$,
since they both describe the boundary of~$E$.
Moreover, $v$ is continuous on~$\overline{\Omega_o}$:
to prove this, suppose by contradiction that
$$ \ell_+ :=\limsup_{{x'\to p'}\atop{x'\in\Omega_o}} u(x') >
\liminf_{{x'\to p'}\atop{x'\in\Omega_o}} u(x') =:\ell_-,$$
for some~$p'\in\overline{\Omega_o}$.
Then~$(x',u(x'))\in (\partial E)\cap\Omega$ and so both~$(p', \ell_-)$
and~$(p',\ell_+)$ belong to~$\partial\big( (\partial E)\cap\Omega \big)$.
As a consequence, sliding~$E_t$ from above as in~\eqref{DEt},
we would obtain that~$t\ge \ell_+-\ell_->0$, in contradiction with~\eqref{t=0}.
Hence the proof of
Theorem~\ref{GRAPH} is complete.
\end{proof}

\section{Smoothness in dimension $3$ and
proof of Theorem \ref{coro:reg}}\label{LJrtyuoo}

The goal of this section is to prove Theorem \ref{coro:reg}.
For this, we state the following result, concerning the boundary
regularity of nonlocal minimal surfaces, which is a direct
consequence of~\cite{PRO}:

\begin{theorem}\label{PRO:TH}
Let~$\Omega$ be an open and bounded subset
of~$\R^{n}$ with boundary of class~$C^{1,\alpha}$,
with~$\alpha\in \left( s+\frac12,\;1\right)$.

Let~$E$ be an~$s$-minimal set in~$\Omega$ and suppose that~$p\in (\partial
\Omega)\cap(\partial E)$.

Assume also that either
\begin{equation}\label{PRO:C-1}
B_r(p)\setminus\Omega\subseteq E
\end{equation}
or
\begin{equation}\label{PRO:C-2}
B_r(p)\setminus\Omega\subseteq E^c,
\end{equation}
for some~$r>0$.

Then, there exists~$r'\in(0,r)$, depending on~$n$, $s$, $\alpha$
and the~$C^{1,\alpha}$ regularity of~$\Omega$, such that~$(\partial E)\cap
B_{r'}(p)$ is of class~$C^{1,\frac12+s}$.
\end{theorem}

\begin{proof} Without loss of generality, we can assume that~$r=2$,
$p=0$ and~\eqref{PRO:C-1} holds true. Then we take~${\mathcal{O}}$
to be a domain with boundary of class~$C^{1,\alpha}$ and such that
$$ B_{1}\cap\Omega^c \subseteq {\mathcal{O}}
\subseteq B_{2}\cap\Omega^c.$$
By construction
\begin{equation}\label{Sp-C-01}
0\in \partial{\mathcal{O}}
\end{equation}
and
\begin{equation}\label{Sp-C-02}
{\mathcal{O}}\cap B_1=B_1 \cap\Omega^c \subseteq B_{2}\cap\Omega^c\subseteq E.\end{equation}
Now we observe that
if~$F$ contains~${\mathcal{O}}\cap B_1$, then
$$ E\cap \Omega^c \cap B_1 = E\cap {\mathcal{O}}\cap B_1
= {\mathcal{O}}\cap B_1 = F\cap {\mathcal{O}}\cap B_1
= F\cap \Omega^c \cap B_1.$$
Also,
if~$F\setminus B_1=
E\setminus B_1$, then
$$ E\cap \Omega^c \cap B_1^c = F\cap \Omega^c \cap B_1^c.$$
Therefore,
if~$F$ contains~${\mathcal{O}}\cap B_1$ and~$F\setminus B_1=
E\setminus B_1$, then
$$ E\cap \Omega^c = (E\cap \Omega^c \cap B_1)\cup(
E\cap \Omega^c \cap B_1^c)= (F\cap \Omega^c \cap B_1)\cup(
F\cap \Omega^c \cap B_1^c)=F\cap \Omega^c,$$
thus, by the minimality of~$E$,
$$ \per(E,\Omega)\le\per(F,\Omega),$$
and therefore
\begin{equation}\label{Sp-C-03}
\per(E,B_1)-\per(F,B_1)=\per(E,\Omega)-\per(F,\Omega)\le0.
\end{equation}
Thanks to~\eqref{Sp-C-01}, \eqref{Sp-C-02} and~\eqref{Sp-C-03},
we can apply Theorem~1.1 in~\cite{PRO} 
and conclude
that~$(\partial E)\cap
B_{r'}$ is of class~$C^{1,\frac12+s}$.
\end{proof}

With this, we are ready to complete the proof of Theorem \ref{coro:reg}.

\begin{proof}[Proof of Theorem \ref{coro:reg}]
By Theorem~\ref{GRAPH},
we know that~$E$ is an epigraph, i.e.~\eqref{aJk:hg1}
holds true for some~$v:\R^{2}\to\R$. It remains to
show that
\begin{equation}
v\in C^\infty(\Omega_o).
\end{equation}
For this, 
we take~$x_o\in (\partial E)\cap\Omega$
and we show that~$v$ is~$C^\infty$ in a neighborhood of~$x_o$.
Up to a translation, we suppose that~$x_o$ is the origin.
Now we consider a blow-up~$E_0$ of the set~$E$, i.e.,
for any~$r>0$,
we define~$E_r:=\frac{E}{r}:=\{ \frac{x}{r} {\mbox{ s.t. }}x\in E\}$
and~$E_0$ to be a cluster point for~$E_r$ as~$r\to0$
(see Theorem~9.2 in~\cite{CRS}). In this way, we have that~$E_0$
is an~$s$-minimal set, and it is an epigraph (see e.g.~(5.8)
in~\cite{figalli}).
Thus, by Corollary~1.3 in~\cite{figalli},
we deduce that~$E_0$ is a half-space.

Hence, by Theorem 9.4 in~\cite{CRS}, we have that~$\partial E$
is a graph of class $C^{1,\alpha}$
in the vicinity\footnote{As a technical remark, we point out that
with the methods of Theorem~9.4 in~\cite{CRS},
one obtains $C^{1,\alpha}$ regularity for any~$\alpha<s$, but this
exponent can be further improved, since
flat minimal surfaces are~$C^{1,\beta}$ for any~$
\beta <1$, with estimates, as proved in Theorem~2.7 of~\cite{PRO}.}
of the origin -- and, as a matter of fact,
of class~$C^\infty$, thanks to
Theorem~1 of~\cite{BEGO}.
\end{proof}

\appendix

\section{Choosing a ``good'' representative for the
$s$-minimal set}\label{9udfdtfd6yh665544}

For completeness, in this appendix we give full details
about the convenient choice of the representative
of an $s$-minimal set.
Indeed, when dealing with an $s$-minimal set, it is useful
to consider a representation of the set which avoids
unnecessary pathologies
(conversely, a bad choice of the set may lead
to the formations of additional boundaries, which come
from subsets of measure zero and can therefore
be neglected). First of all, one would like to
chose the representative of the set with ``the smallest
possible boundary''. We will also show that
one can reduce the analysis to an open set, by
considering the points of the set which are interior
in the sense of measures, according to the following observation:

\begin{lemma}\label{LE:LAIH0}
Let $\Omega$ be an open subset of~$\R^n$, with~$|\partial\Omega|=0$
and~$E$ be an $s$-minimal set in~$\Omega$.
Assume that~$E\setminus\Omega$ is open in~$\R^n\setminus\Omega$.
Let
\begin{equation}\label{8udGGHJ:JH}\begin{split}
E_1\,&:= \big\{ x\in\R^n {\mbox{ s.t. there exists }} r>0 {\mbox{ s.t. }}
|E\cap B_r(x)|=|B_r(x)| \big\}\\
&= \big\{ x\in\R^n {\mbox{ s.t. there exists }} r>0 {\mbox{ s.t. }}
|B_r(x)\setminus E|=0 \big\}.\end{split}\end{equation}
Then:
\begin{eqnarray}
\label{L678-0}
&& {\mbox{$E_1$ is an open set,}}\\
\label{L678-1}
&& {\mbox{$E$ and $E_1$ coincide, up to a set of measure zero,}}\\
\label{L678-2}
&& {\mbox{there exists~$c\in(0,1)$ such that for any $x\in\partial E_1$
and any $r>0$ for which $B_r(x)\subset\Omega$}}\\
&& {\mbox{there exist~$y_1$, $y_2\in B_r(x)$ such that~$B_{cr}(y_1)
\subset B_r(x)\cap E_1$ and~$B_{cr}(y_2)
\subset B_r(x)\setminus E_1$,}} \nonumber \\
\label{L678-3}
&& {\mbox{for any $x\in\partial E_1$, }}
0 < \frac{|E_1\cap B_r(x)|}{|B_r(x)|}<1.
\end{eqnarray}
\end{lemma}

\begin{proof}
The statement in~\eqref{L678-0} comes directly from~\eqref{8udGGHJ:JH},
so we focus on the proof of~\eqref{L678-1},
\eqref{L678-2} and~\eqref{L678-3}. The proof of these facts
is a consequence of the density estimates 
of $s$-minimal sets, see~\cite{CRS}.
We provide full details for the convenience of the reader.

First of all, we can reduce from measurable sets
to Borel sets, up to sets of measure zero, see e.g.
Theorem~2.20(b) in~\cite{Rudin}. Hence we may suppose that~$E$
is Borel. Also, we observe that, since~$E\setminus\Omega$
is relatively open in~$\R^n\setminus\Omega$, we have that
\begin{equation}\label{B3-sdtgf556}
{\mbox{$E_1\setminus
\overline\Omega$ coincides with~$E\setminus
\overline\Omega$.}}\end{equation} Also, the symmetric difference of~$E_1$ and~$E$
restricted to~$\partial\Omega$ has zero measure, since so does~$\partial\Omega$,
therefore the proof of~\eqref{L678-1} is nontrivial
only due to the possible contributions inside~$\Omega$.

Now we set
\begin{eqnarray*}
E_0&:=& \big\{ x\in\R^n {\mbox{ s.t. there exists }} r>0 {\mbox{ s.t. }}
|E\cap B_r(x)|=0 \big\},\\
{\mbox{and }}\;
\tilde E &:=& (E\cup E_1)\setminus E_0.
\end{eqnarray*}
By Proposition~3.1
in~\cite{Giusti}, one has that~$\tilde E$
coincides with~$E$ up to sets of measure zero, that
\begin{equation}\label{78Gh4r-1ujs}
|E\cap E_0|=
|E_1\setminus E|=0,\end{equation}
and
\begin{equation}\label{ZZIONE}
0 < \frac{|E\cap B_r(x)|}{|B_r(x)|}<1\end{equation}
for any~$x\in \partial E$ and any~$r>0$.

Also, by construction, we see that~$E_0\cap E_1=\varnothing$
and so
\begin{equation}\label{e1-po}
E_1\subseteq\tilde E.\end{equation}

We set~$L$ to be the set of Lebesgue points of~$\tilde E$, i.e.
the set of all points~$x\in \tilde E$ such that
$$ \lim_{r\to0} \frac{|\tilde E\cap B_r(x)|}{|B_r(x)|}=1.$$
We recall (see e.g. Theorem~7.7 in~\cite{Rudin})
that
\begin{equation}\label{87yGGH789psp-0}
|\tilde E\setminus L|=0.\end{equation}
Now we claim that
\begin{equation}\label{87yGGH789psp}
\big( \tilde E\setminus (E_0\cup E_1)\big)\cap\Omega\subseteq \tilde E\setminus L.
\end{equation}
To prove this, let~$x\in \big(\tilde E\setminus (E_0\cup E_1)\big)\cap\Omega$.
Then, for any~$r>0$, we have that~$|E\cap B_r(x)|>0$
and~$|E^c \cap B_r(x)|>0$, 
and so the same inequalities hold for~$\tilde E$ replacing~$E$.

Accordingly, for any~$r>0$ such that~$B_r(x)\subset\Omega$, there exists~$p_r\in B_r(x)\cap(\partial \tilde E)$.
By~\eqref{ZZIONE} and the Clean Ball Condition in Corollary~4.3
of~\cite{CRS}, we deduce that there exist
points~$y_{1,r}$ and~$y_{2,r}$
such that~$B_{cr}(y_{1,r})
\subset B_r(p_r)\cap \tilde E$ and~$B_{cr}(y_{2,r})
\subset B_r(p_r)\setminus \tilde E$, for a universal~$c\in(0,1)$.

Since~$B_r(p_r)\subseteq B_{2r}(x)$,
we see that~$B_{cr}(y_{1,r})
\subset B_{2r}(x)\cap \tilde E$ and~$B_{cr}(y_{2,r})
\subset B_{2r}(x)\setminus \tilde E$, and therefore
$$  \frac{|\tilde E\cap B_{2r}(x)|}{|B_{2r}(x)|}
\le \frac{|B_{2r}(x)\setminus B_{cr}(y_{2,r})|}{|B_{2r}(x)|}
= \frac{2^n-c^n}{2^n}.$$
This implies that
$$ \limsup_{r\to0}
\frac{|\tilde E\cap B_{2r}(x)|}{|B_{2r}(x)|}
\le \frac{2^n-c^n}{2^n} <1,$$
and so~$x\not\in L$, which proves~\eqref{87yGGH789psp}.

By \eqref{B3-sdtgf556},
\eqref{87yGGH789psp} and~\eqref{87yGGH789psp-0}, we obtain that~$
|\tilde E\setminus (E_0\cup E_1)|=0$.
So, by~\eqref{78Gh4r-1ujs},
$$ |\tilde E\setminus E_1|\le |\tilde E\setminus (E_0\cup E_1)|+
|\tilde E\cap E_0|=0+|E\cap E_0|=0.$$
Moreover, by~\eqref{78Gh4r-1ujs},
we also have that
$$ |E_1\setminus\tilde E|=|E_1\setminus E|=0,$$
therefore~$E_1$ and~$\tilde E$ agree up to a set of measure
zero. This establishes~\eqref{L678-1}.

Now we prove~\eqref{L678-2}. For this, let~$x\in \partial E_1$
and suppose that~$B_r(x)\subset\Omega$.
By construction, there exist $a_r\in B_{r/4}(x)\cap E_1$
and~$b_r \in B_{r/4}(x)\setminus E_1$. 
So, by~\eqref{e1-po}, we have that
\begin{equation}\label{98ytsdcUUkaa-09}
a_r \in B_{r/4}(x)\cap\tilde E.
\end{equation}
Also, by construction, 
$$ |B_{r/4}(b_r)\setminus \tilde E|=|B_{r/4}(b_r)\setminus E|>0,$$
therefore there exists~$c_r\in B_{r/4}(b_r)\setminus \tilde E$,
and so in particular~$c_r\in B_{r/2}(x)\setminus \tilde E$.

{F}rom this and~\eqref{98ytsdcUUkaa-09}, we obtain that there exists~$q_r
\in B_{r/2}(x)\cap (\partial \tilde E)$.
Then, by the Clean Ball Condition in Corollary~4.3
of~\cite{CRS}, we deduce that there exist
points~$y_1$ and~$y_2$
such that
$$ B_{cr/2}(y_1)
\subset B_{r/2} (q_r)\cap \tilde E
\subset B_r (x)\cap \tilde E$$
and
$$ B_{cr/2}(y_2)
\subset B_r(q_r)\setminus \tilde E
\subset B_r (x)\setminus \tilde E,$$
for a universal~$c\in(0,1)$.
This says that~$B_{cr/2}(y_1)$ lies in~$E$
and $ B_{cr/2}(y_2)$ lies in~$E^c$, up to sets of measure zero,
therefore~$B_{cr/2}(y_1) \subseteq E_1$
and $ B_{cr/2}(y_2)\subseteq E_0\subseteq E_1^c$.
This completes the proof of~\eqref{L678-2}
(up to changing~$c/2$ to~$c$).

Now, \eqref{L678-3}
is a direct consequence of~\eqref{L678-2}.
\end{proof}

Thanks to Lemma~\ref{LE:LAIH0}, when
dealing with the $s$-minimal set~$E$ in the statement of
Theorems \ref{GRAPH} and~\ref{coro:reg},
we implicitly identify~$E$ with~$E_1$.

\section*{Acknowledgments}
The first author has been supported by EPSRC grant  EP/K024566/1
``Monotonicity formula methods for nonlinear PDEs'', ERPem
``PECRE Postdoctoral and Early Career Researcher Exchanges'' and Alexander von Humboldt Foundation.
The second author has been supported by
NSF grant DMS-1200701.
The third author has been supported by ERC grant 277749 ``EPSILON Elliptic
Pde's and Symmetry of Interfaces and Layers for Odd Nonlinearities''
and PRIN grant 201274FYK7
``Aspetti variazionali e
perturbativi nei problemi differenziali nonlineari''.

\end{document}